\documentclass[11pt,a4paper,leqno]{amsart}

\usepackage[latin1]{inputenc}
\usepackage[T1]{fontenc}
\usepackage{amsfonts}
\usepackage{amsmath}
\usepackage{amssymb}
\usepackage{eurosym}
\usepackage{mathrsfs}
\usepackage{palatino}
\usepackage{color}
\usepackage{xcolor}
\usepackage{esint}
\usepackage{url}
\usepackage{graphicx}
\usepackage{hyperref} %,hypertexnames=false,colorlinks,[pagebackref]
\hypersetup{
	colorlinks=true,       % false: boxed links; true: colored links
	linkcolor=blue,          % color of internal links
	citecolor=magenta,        % color of links to bibliography
	filecolor=magenta,      % color of file links
	urlcolor=cyan           % color of external links
}

\numberwithin{equation}{section}

\newcommand{\abs}[1]{|#1|}

\newcommand{\BMO}[0]{\operatorname{BMO}}

\theoremstyle{plain}
\newtheorem{thm}{Theorem}[section]
\newtheorem{lem}[thm]{Lemma}
\newtheorem{prop}[thm]{Proposition}

\theoremstyle{definition}
\newtheorem{defn}[thm]{Definition}

\theoremstyle{remark}
\newtheorem{rem}[thm]{Remark}

\title{Off-diagonal bloom weighted estimates for bilinear commutators }

\author{Yunan Zeng}

\address{Center for Applied Mathematics, Tianjin University, Weijin Road 92, 300072 Tianjin, China}
\email{zynn@tju.edu.cn}

\makeatletter
\@namedef{subjclassname@2020}{%
	\textup{2020} Mathematics Subject Classification}
\makeatother


\begin{document}
	
	\allowdisplaybreaks
	
	\begin{abstract}
		We prove the off-diagonal estimates of the bilinear iterated commutators in the two-weight setting. The upper bound is established via sparse domination, and the lower bound is proved by  the median method. Our methods are so flexible so that it can be easily extended to the multilinear scenario.
	\end{abstract}

	\maketitle
	
	\section{Introduction}
	In recent years the topic of commutators has attracted a lot of attention, due to its nice applications such as elliptic PDEs and Hardy spaces (see e.g. \cite{MR1088476, MR412721}). Given a linear operator $T$ and a locally integrable  function  $b$, recall that the commutator of $T$ and $b$ is defined as
	\begin{equation}
		[b,T]f=bT(f)-T(bf).\nonumber
	\end{equation}
	The problem that we are concerned with is the characterization of the boundedness of $[b, T]$ in terms of certain membership of the function $b$. When $T$ is a Calder\'on-Zygmund operator, the candidate of $b$ is the so-called bounded mean oscillation function space, which is usually denoted by $\BMO$. Recall that we say $b\in \BMO$ if 
	\[
	\|b\|_{\BMO}:= \sup_Q \frac 1{|Q|}\int_Q \big|b-\langle b\rangle_Q\big|<\infty,
	\]where $\langle b\rangle_Q= \frac 1{|Q|}\int_Q b$. 
	Initially, Nehari \cite{MR82945} proposed a complex analysis method to attack this problem for the Hilbert transform. In 1976, Coifman, Rochberg and Weiss \cite{MR412721} studied this problem for the Riesz transform using real analysis method. They provided the following commutator lower and upper bound
	\begin{equation}\label{key 11}
		\|b\|_{\BMO(\mathbb{R}^{n})}\lesssim\sum_{j=1}^{n} \|[b,\mathcal{R}_j]\|_{L^{p}(\mathbb{R}^n)\to L^{p}(\mathbb{R}^n)}\lesssim \|b\|_{\BMO(\mathbb{R}^n)}.
	\end{equation}
	The lower bound was also addressed by Janson \cite{MR524754} and Uchiyama \cite{MR467384} independently. They extended the formula (\ref{key 11}) to a wider class of singular integrals that includes any single Riesz transform. Janson's proof also gave the  off-diagonal characterization of the boundedness of the commutator. When $1<p<q<\infty$, there holds that
	\begin{equation}
		\|[b,T]\|_{L^p\to L^q}\simeq \|b\|_{\dot{C}^{0,\alpha}}:=\sup_{x\neq y}\frac{\abs{b(x)-b(y)}}{\abs{x-y}^{\alpha}},\quad \frac{\alpha}{n}=\frac{1}{p}-\frac{1}{q}.\nonumber
	\end{equation}
	The case $p>q$ was left open there. It was until 2021, Hyt\"onen finally completed the picture in \cite{MR4338459}, providing a full characterization of $\|[b,T]\|_{L^p\to L^q}$, where $T$ is a non-degenerate Calder\'on-Zygmund operator. In \cite{MR4338459}, it is defined that a Calder\'on-Zygmund operator $T$ is called non-degenerate if its kernel $K$ satisfies that, for every $y\in \mathbb{R}^n$ and $r>0$, there exists $x\in B(y,r)^{c}$ with
	$$\abs{K(x,y)}\geqslant \frac{1}{c_0r^d}.$$
	The main contribution of this paper is that it establishes a lower bound estimate by median and approximate weak factorization methods.

	Now an interesting problem is to study the two weight counterpart of the above question. That is, how to characterize the boundedness of $[b, T]$ from $L^p(\mu)$ to $L^p(\lambda)$, where $\mu$, $\lambda$ are non-negative locally integrable functions, and the candidate for the function $b$ is $\BMO_\nu$ with suitable $\nu$. 
	Here, the weighted $\BMO$ space,  $\BMO_\nu$ with $\nu\in A_{\infty}$ was introduced by Muckenhoupt and Wheeden \cite{MR399741}, recall that $b\in \BMO_\nu$ if 
	\[
	\|b\|_{\BMO_\nu}:= \sup_Q \frac 1{\nu(Q)}\int_Q \big|b-\langle b\rangle_Q\big|<\infty. 
	\]
	This topic was initiated by Bloom \cite{MR805955} in 1985, who considered this question for the Hilbert transform. He proved that
	$$\big\|[b,H]\big\|_{L^p(\mu)\to L^p(\lambda)}\simeq \|b\|_{\BMO_\nu},$$
	where $1<p<\infty, \mu,\lambda\in A_p, \nu=(\mu/\lambda)^{\frac{1}{p}}$. The weight $\nu$ is referred as the Bloom weight. Segovia and Torrea \cite{MR1074151} extended Bloom's upper bound to general Calderón-Zygmund operators. This issue regained attention because  Holmes, Lacey and Wick
	 \cite{MR3451366, MR1029684} revisited the result using the modern technology of dyadic representation theorems. They also extended Bloom's lower bound
	 result to the Riesz transform.
     Only shortly after, Lerner, Ombrosi, and Rivera-R\'ios \cite{MR3695871} gave a quantitative upper bound estimate using sparse domination, which is a simpler proof. In particular, they provided the following pointwise sparse domination principle 
	$$\abs{[b,T]f}\lesssim \sum_{Q\in \mathcal{S}}\abs{b-\langle b \rangle_Q}\big\langle \abs{f} \big\rangle_Q \chi_Q+\sum_{Q\in \mathcal{S}}\big\langle\abs{b-\langle b\rangle_Q}\abs{f}\big\rangle_Q\chi_Q,$$
	where $\mathcal{S}$ is a sparse family (see Section \ref{sec:pre} for the definition). The Bloom's lower bound for general case, i.e. non-degenerate singular integrals operators, was proved by Hyt\"onen \cite{MR4338459}.
	Meanwhile, the off-diagonal boundedness of the commutator were also developed. The weighted analogies of the off-diagonal estimates were given by \cite{MR4587286} for exponents  $1<p< q<\infty$ and  \cite{Hanninen2023WeightedL} for exponents  $1<q<p<\infty$ with $T$ was a non-degenerate Calder\'on-Zygmund operator.
	Recently, Lerner, Lorist and Ombrosi \cite{MR4718669} further explored the $L^p$ to $L^q$ quantitative Bloom weighted estimates for iterated commutators of a class of operators which are more general than the Calder\'on-Zygmund setting, where $1<p,q<\infty$. Here the ($k$th-order) iterated commutators are defined inductively as
	$$\mathcal{C}_{b}^k(T):=[b,\mathcal{C}_b^{k-1}(T)],\quad \mathcal{C}_b^1(T):=[b,T],\qquad k\in \mathbb{N}_+. $$
	They also developed a sparse domination principle in form sense, that is 
	$$\big\langle \mathcal{C}_b^{k}(T)f,g\big\rangle \lesssim  \sum_{Q\in \mathcal{S}}\big\langle \abs{b-\langle b\rangle_Q}^k\abs{f}\big\rangle_{r,Q}\big\langle \abs{g}\big\rangle_{s^\prime,Q}\abs{Q}+\sum_{Q\in \mathcal{S}} \big\langle\abs{f}\big\rangle_{r,Q}\big\langle \abs{b-\langle b\rangle_Q}^k\abs{g}\big\rangle_{s^{\prime},Q} \abs{Q},$$
	where $\mathcal{S}$ is a sparse family. For more about the linear theory, we refer the readers to \cite{MR3926046, MR4167272} and the references therein.
	
	It is also natural to extend the definition of iterated commutators to the multilinear scenario. Given an $n$-linear operator $T$ and $b \in \ L^{1}_{\rm{loc}}(\mathbb{R}^n)$, we define 
	$$[b,T]_1(f_1,f_2,\ldots, f_n):=bT(f_1,f_2,\ldots, f_n)-T(bf_1,f_2,\ldots, f_n),$$
	one my define $[b, T]_j$ for other values of $j\in \{1,\ldots, n\}$ in a similar way.
	It is easy to see that $[b,[b,T]_i]_j=[b,[b,T]_j]_i$. As a result we only  consider the  iterated commutator 
	$$  \mathcal{C}^{k_1}_{b}(T)=[b,\mathcal{C}_b^{k_1-1}(T)]_1,\quad  \mathcal{C}^{1}_{b}(T)=[b, T]_1. $$
	In 2009,
	Lerner et al. \cite{MR2483720}  proved the commutators estimate in the one-weight situation, namely 
	$$[b,T]_1:L^{p_1}(\omega_1^{p_1})\times \dots \times L^{p_n}(\omega_n^{p_n})\to L^{p}(\prod_{i=1}^{n}\omega_i^{p}),$$
	where $(\omega_1,\dots,\omega_n)\in A_{\vec{p}}$. In \cite{MR3874959}, Kunwar and Ou investigated the boundedness of multilinear commutators in two-weight setting. However, they did not use the genuinely multilinear weights.
	Li \cite{MR4453692} improved this result to genuinely multilinear weights. For further reading on multilinear theory, we refer the readers to \cite{MR3573728,MR4167272, MR4647944}.
	%$$\mathcal{C}^{k_1}_{b}(T):L^{p_1}(\omega_1^{p_1})\times \dots \times L^{p_n}(\omega_n^{p_n})\to L^{p}(\lambda_1^{p}\prod_{j=2}^n\omega_j^{p}),$$
  	%where $(\omega_1,\dots,\omega_n),(\omega_1,\dots,\lambda_i,\dots,\omega_n)\in A_{\vec{p}}$ .
	
	Motivated by \cite{MR4718669} and \cite{MR4453692}, in this paper we extend the off-diagonal two weight estimates for the linear non-iterated commutators to the case of bilinear iterated commutators.
	Our main result states as the following: 
	\begin{thm}\label{thm 10}
		Let  $ 1\leqslant r_1 <p_1,q_1< \infty, 1\leqslant r_2 <p_2 <\infty, 1<p,q<s\leqslant \infty$, satisfy 
		$ {1}/{p}={1}/{p_1}+{1}/{p_2}, {1}/{q}={1}/{q_1}+{1}/{p_2}, $
		and $k_1\in \mathbb{N}$. Let $T$ be a bilinear operator and $b\in L_{\rm{loc}}^1(\mathbb{R}^n)$. Suppose that
		$T$ and $\mathcal{M}^{\#}_{T,s}$ are bilinear locally weak $L^{(r_1,r_2)}$-bounded and $(\lambda_1, \omega_2 )\in A_{\vec{q},\vec{r}}, (\omega_1,\omega_2)\in A_{\vec{p},\vec{r}}$, where $\vec{q}=(q_1,p_2),\vec{p}=(p_1, p_2 ),\vec{r}=(r_1,r_2,s^\prime)$. Set
		\begin{equation}\label{eq:defnualpha}
			\nu_1^{1+\alpha_1}:=\omega_1^{\frac{1}{k_1}}\lambda_1^{-\frac{1}{k_1}}, \quad \alpha_1:=\frac{1}{p_1k_1}-\frac{1}{q_1k_1}
		\end{equation}
		and $\nu_1\in A_{\infty}$. Then:
		\begin{enumerate}
			\item  If $p_1\leqslant q_1$, we have
			$$\big\|C_{b}^{k_1}(T)\big\|_{L^{p_1}(\omega_1^{p_1})\times L^{p_2}(\omega_2^{p_2})\rightarrow L^q((\lambda_1 \omega_2)^{q})} \lesssim_{\omega_1,\omega_2,\lambda_1}	\|b\|_{BMO_{\nu_1}^{\alpha_1}}^{k_1}.$$
			\item If $q_1< p_1$, we have
			$$\big\|C_{b}^{k_1}(T)\big\|_{L^{p_1}(\omega_1^{p_1})\times L^{p_2}(\omega_2^{p_2})\rightarrow L^q((\lambda_1 \omega_2)^{q})} \lesssim_{\omega_1,\omega_2,\lambda_1}\|M_\nu ^\#(b)\|_{L^{t_1}(\nu_1)}^{k_1},$$
			where $t_1=-1/\alpha_1$.
		\end{enumerate}
	\end{thm}
	\begin{rem}
		If $p=q$ and $\omega_1=\lambda_1$, with the help of \cite[Theorem 1.1]{MR4129475} we can extend the result to $p=q<1$. Similarly, if $\vec{r}=(1,1,1)$  and $p=q$, by invoking \cite[Theorem 1.3]{MR4684375} one can extend the result to $p=q<1$ as well. In general if there is an extrapolation theorem that holds for off-diagonal two weight estimates, we can reach the quasi-Banach range. 
	\end{rem}
	\begin{thm}\label{thm 2}
		Let $1<p_1, p_2,q_1<\infty, 1/{p}=1/{p_1}+1/{p_2}, 1/{q}=1/{q_1}+1/{q_2}$ and  $k_1\in \mathbb{N}$. Let $T$ be a bilinear non-degenerate Calder\'on-Zygmund operator and $b$ be a real-valued $L_{\rm{loc}}^{k_1}(\mathbb{R}^n)$ function. Assume that $(\omega_1, \omega_2)\in A_{\vec{p}}, (\lambda_1, \omega_2)\in A_{\vec{q}}$, where $\vec{q}=(q_1,p_2),\vec{p}=(p_1, p_2 )$, and let $\nu_1, \alpha_1$ be defined as in \eqref{eq:defnualpha}. Then
		\begin{enumerate}
			\item  If $p_1\leqslant q_1$, we have
			\begin{equation}
				\|b\|^{k_1}_{BMO_{\nu_1}^{\alpha_1}}\lesssim \Big\|C_{b}^{k_1}(T):L^{p_1}(\omega_1^{p_1})\times L^{p_2}(\omega_2^{p_2})\to L^{q,\infty}((\lambda_1\omega_2)^{q})\Big\|\nonumber.
			\end{equation}
			\item  If $q_1< p_1$ and set $t_1=-1/{\alpha_1}$, there holds 
			\begin{equation}
				\|M_{\nu_1} ^\#(b)\|^{k_1}_{L^{t_1}(\nu_1)}\lesssim \Big\|C_{b}^{k_1}(T):L^{p_1}(\omega_1^{p_1})\times L^{p_2}(\omega_2^{p_2})\to L^{q}((\lambda_1\omega_2)^{q})\Big\|\nonumber.
			\end{equation}
		\end{enumerate}
		\end{thm}
	The exact definitions of these notations will be provided in Section \ref{sec:pre} and Section \ref{sec:pro}. 
	The upper bound is similar to \cite{MR4718669}. The main difficulty encountered is that the linear method in \cite{MR4718669} is no longer applicable in bilinear situation. Here, we establish a new sparse domination to overcome this difficulty. The lower bound is solved by using median method and some weight conditions.

	This paper is organized as the following: In Section \ref{sec:pre}, we give some notations and definitions. A sparse domination of certain bilinear operator is provided in Section \ref{sec:spr}. In Section \ref{sec:we}, we extend the results of Li  \cite{MR3591468} to obtain weighted estimates for sparse the form. Section \ref{sec:pro} is denoted to proving Theorem \ref{thm 10} and Theorem \ref{thm 2}.
	
	\section{Preliminaries}\label{sec:pre}
	\subsection{Dyadic cube system}
	We denote by $\mathcal{Q}$ the set of all cubes $Q\subset\mathbb{R}^n$ with sides parallel to the axes. Given a cube $Q\subset \mathcal{Q}$, we use $\mathcal{D}(Q)$ to denote the set of all dyadic cubes with respect to $Q$, that is, the cubes obtained by repeated subdivision of $Q$ and each of its descendants into $2^n$ congruent subcubes.   
	A dyadic lattice $\mathscr{D}$ is a collection of cubes of $\mathbb{R}^n$ such that:
	\begin{enumerate}
		\item If $Q\in \mathscr{D}$, then $\mathcal{D}(Q)\subset \mathscr{D}$.
		\item Any $Q^\prime, Q^{\prime\prime} \in \mathscr{D}$ have a common ancestor, namely, we can find $Q\in \mathscr{D}$ such that $Q^\prime, Q^{\prime\prime} \in \mathcal{D}(Q)$.
		\item For any compact set $K\subset \mathbb{R}^n$, there exists a cube $Q\in \mathscr{D}$ containing $K$.
	\end{enumerate}
	\subsection{Basic notations}Given a cube $Q\in \mathcal{Q}$, a measure $\mu$ and $f\in{L^{r}_{\rm{loc}}(\mathbb{R}^n)}$, we denote the average  $$\langle f \rangle_{r,Q}^{\mu}:=\Big(\frac{1}{\mu (Q)} \int_{Q}f^{r}d\mu\Big)^{\frac{1}{r}}.$$
	 When $\mu$ is Lebesgue measure we omit $\mu$ in the supscript and simply write it as $\langle f \rangle_{r,Q}$. In particular, when $r=1$, we set  $\langle f \rangle_{Q}:= \langle f \rangle_{1,Q}=\frac{1}{|Q|} \int_{Q}f.$ 
	We denote the maximal operator and bi-maximal operator as follows 
	\begin{align}
	&M_{r,\mu}(f):=\sup \limits _{Q\in \mathcal{Q}} \langle |f| \rangle ^{\mu}_{r,Q} \chi_{Q},\nonumber\\
	&M_{\vec{r}}(f):=\sup \limits _{Q\in \mathcal{Q}} \prod_{i=1}^{2}\langle |f_i| \rangle_{r_i,Q} \chi_{Q},\nonumber
	\end{align}
    where $\vec{r}=(r_1,r_2)$. We also denote "$A\lesssim B$", when $A\leqslant CB$, where $C$ is some irrelevant constant. If these $C$ depends on the weight $\mu$, this will be denoted by "$A\lesssim_{\mu}B$". We use $\ell(Q)$ to denote the side length and use $CQ$ to denote a cube with the same center as $Q$ and the side length is $C$ times of $Q$. 
	Let $\eta\in(0,1)$. A collection $\mathcal{S}$ of cubes is called $\eta$-sparse if for each $Q \in \mathcal{S}$, there exists a subset $E_Q \subset Q$ such that $|E_Q| \geqslant \eta|Q|$ and the sets $\{E_Q\}_{Q \in \mathcal{S}}$ are pairwise disjoint.
		\subsection{Weight}
		We say a weight $\omega\in A_p$ if
		$$[\omega]_{A_p}:=\sup\limits_{Q\in \mathcal{Q}}\langle \omega \rangle_{Q}\langle \omega^{1-p^\prime} \rangle_{Q}^{p-1}<\infty, \quad  1<p<\infty.$$
		Furthermore, we say that $\omega \in A_{\infty}$ if
		$$[\omega]_{A_\infty}:= \sup\limits_{Q\in \mathcal{Q}}\frac{1}{\omega(Q)}\int_{Q}M(\omega\chi_{Q})<\infty.$$
		The following properties of $A_\infty$ weights are crucial for us: 
		\begin{lem}\emph{(}\cite[Lemma 2.5]{MR4453692}\emph{)}\label{key lem1}
			Let $t_1,\dots,t_n\in(0,\infty)$ and $\omega_1,\dots,\omega_n\in A_{\infty}$. Then there exists a constant $RH_{\vec{\omega},\vec{t}}$ such that for every cube $Q$,
			\begin{equation}
				\prod_{i=1}^{n}\big(\langle \omega_i \rangle_{Q} \big) ^{t_i} \leq RH_{\vec{\omega},\vec{t}}\, \Big\langle \prod_{i=1}^{n}\omega_i ^{t_i}\Big\rangle_Q.\nonumber
			\end{equation}
		\end{lem}
		\begin{lem}\emph{(}\cite[Lemma 2.9]{MR4453692}\emph{)}\label{key lem2}
			Let $t_1,\dots,t_n\in(0,\infty)$ and $\omega_1,\dots,\omega_n$ be weights such that $\prod_{i=1}^{n}\omega_i^{t_i} \in A_{\infty}$. Then there exists a constant K depending on $[\prod_{i=1}^{n}\omega_i^{t_i}]_{A_{\infty}}$ such that for every cube $Q$,
			\begin{equation}
				\Big\langle \prod_{i=1}^{n}\omega_i ^{t_i}\Big\rangle_Q \leq K\prod_{i=1}^{n}\big(\langle \omega_i \rangle_{Q} \big) ^{t_i} .\nonumber
			\end{equation}
		\end{lem}
		
		In the bilinear setting, for $1< p_1,p_2<\infty, 1/{p}=1/{p_1}+1/{p_2}$ and $\vec{p}=(p_1,p_2)$, we say that $\vec{\omega} := (\omega_1,\omega_2)\in A_{\vec{p}}$ if
		$$[\vec{\omega}]_{A_{\vec{p}}}  :=\sup\limits_{Q\in \mathcal{Q}} \langle \omega \rangle _{p,Q}\langle \omega_1^{-1} \rangle _{p_1^\prime,Q} \langle \omega_2^{-1} \rangle _{p_2^\prime,Q} <\infty, \quad \omega:=\omega_1\omega_2. $$
		It is proved in \cite{MR2483720} that $\vec{\omega} \in A_{\vec{p}}$
		if and only if
		\begin{align}
			\omega^{p}\in A_{2p},\quad   \omega_1^{-p_1^\prime}\in A_{2p_1^\prime},\quad \omega_2^{-p_2^\prime}\in A_{2p_2^\prime}.\nonumber
		\end{align}
		Li, Martell and Ombrosi generalized the $A_{\vec{p}}$ weights in \cite{MR4129475} to the more general   $A_{\vec{p},\vec{r}}$ weights. Recall that for $1\leqslant r_1\leqslant p_1< \infty, 1\leqslant r_2\leqslant p_2< \infty,  p<r^\prime_3\leqslant \infty$ and $1/{p}=1/{p_1}+1/{p_2}$, we say that $\vec{\omega} := (\omega_1,\omega_2)\in A_{\vec{p},\vec{r}}$ if
		$$[\vec{\omega}]_{A_{\vec{p},\vec{r}}} :=\sup\limits_{Q}\big\langle \omega \big\rangle _{\frac{r_3^\prime p}{r_3^\prime-p},Q}\big\langle \omega_1^{-1}\big\rangle_{\frac{r_1p_1}{p_1-r_1},Q}\big\langle \omega_2^{-1}\big\rangle_{\frac{r_2p_2}{p_2-r_2},Q} <\infty, \quad \omega:=\omega_1\omega_2.$$
		In the case of $r_3=1$, the first term  $\langle \omega \rangle _{{r_3^\prime p}/{r_3^\prime-p},Q}$ is understood as $\langle \omega \rangle_{p,Q}$. Analogously, when $p_i=r_i$, the terms corresponding to $\omega_i$ needs to be replaced by $\rm{ess}\sup_{Q}\omega_{i}^{-1}$. It should be mentioned that $A_{\vec{p},(1,1,1)} = A_{\vec{p}}$. It is proved in \cite{MR4129475} that $\vec{\omega}\in A_{\vec{p},\vec{r}}$ if and only if
		$$\omega^{\delta_3}\in A_{\frac{1-r}{r}\delta_3},\quad \omega_1^{\theta_1}\in A_{\frac{1-r}{r}\theta_1}, \quad \omega_2^{\theta_2}\in A_{\frac{1-r}{r}\theta_2},$$
		where
		\begin{align}
			\frac{1}{r}:=\sum_{i=1}^{3}\frac{1}{r_i},\quad \frac{1}{p_3}:=1-\frac{1}{p},\quad\frac{1}{\delta_i}:=\frac{1}{r_i}-\frac{1}{p_i},\quad i=1,2,3,\nonumber
		\end{align}
		and
		\begin{align}
			\frac{1}{\theta_i}:=\frac{1-r}{r}-\frac{1}{\delta_i},\quad i=1,2,3.\nonumber
		\end{align}

		\subsection{Bilinear Calder\'on-Zygmund operators}
		Let $\Delta:=\big\{(x,y_1,y_2)\in (\mathbb{R}^n)^3: x=y_1=y_2\big\}$. We say that $K:(\mathbb{R}^n)^3\setminus \Delta \to \mathbb{C}$ is bilinear Calder\'on-Zygmund kernel if
		\begin{equation}
			|K(x,y_1,y_2)|\leqslant \frac{c_K}{(|x-y_1|+|x-y_2|)^{2n}},
		\end{equation}
		and whenever $h\leqslant\frac{1}{2}\operatorname{max}(|x-y_1|,|x-y_2|)$, 
		\begin{equation}\label{key formular9}
			\begin{split}
			|K(x+h,y_1,y_2)&-K(x,y_1,y_2)|+|K(x,y_1+h,y_2)-K(x,y_1,y_2)|\\
			&+|K(x,y_1,y_2+h)-K(x,y_1,y_2)|\\
			\leqslant&\frac{1}{(|x-y_1|+|x-y_2|)^{2n}}\omega\Big(\frac{h}{|x-y_1|+|x-y_2|}\Big),
			\end{split}
		\end{equation}
		where $\omega$ is a continuous modulus, that is an increasing subadditive function with $\omega(0)=0$ and
		$$\|\omega\|_{\rm{Dini}} = \int_{0}^{1}\omega(t)\frac{dt}{t}< \infty.$$
		We say that $T$ is a bilinear Calder\'on-Zygmund operator, if there exists a bilinear Calder\'on-Zygmund kernel $K$ such that
		\begin{align}
			T(f_1,f_2)(x)=\int_{\mathbb{R}^{2d}}K(x,y_1,y_2)f_1(y_1)f_2(y_2)dy, \quad x\notin \rm{spt}f_1 \cap \rm{spt}f_2\nonumber,
		\end{align}
		and for some $1<p_1,p_2<\infty$ with $ {1}/{p}={1}/{p_1}+{1}/{p_2}$, $T$ is $L^{p_1}(\mathbb{R}^n)\times L^{p_2}(\mathbb{R}^n) \to L^p(\mathbb{R}^n)$ bounded.
		
		We also call a bilinear Calder\'on-Zygmund kernel $K$ is non-degenerate, if for every $y\in \mathbb{R}^d$ and $r>0$, there exists $x \notin B(y,r)$ such that
		\begin{equation}
			|K(x,y,y)|\geq \frac{1}{c_0r^{2d}}.\label{formula 5}
		\end{equation}
	
		\section{A sparse domination principle}\label{sec:spr}
		For $1\leqslant {s}\leqslant{\infty}$, we define the bilinear sharp  grand truncation maximal operator, which is a modification of definition in $\mathcal{M}^{\#}_{T,s}(\vec{f} )(x)$ \cite{MR4058547},
		\begin{align}
			\mathcal{M}^{\#}_{T,s}(f_1,f_2)(x)&:=\sup\limits_ {Q\ni x}\Big(\dfrac{1}{|Q|^{2}}\int_{Q\times Q} \big|\big(T(f_1,f_2)-T(f_1\chi_{3Q},f_2\chi_{3 Q})\big)(x^{\prime})\nonumber\\
			&-\big(T(f_1,f_2)-T(f_1\chi_{3 Q},f_2\chi_{3 Q})\big)(x^{\prime \prime})\big|^{s}dx^{\prime}dx^{\prime \prime}\Big)^{\frac{1}{s}},\nonumber
		\end{align}
		where the supremum is taken over all $Q\in \mathcal{Q}$ containing $x$.
		
		\subsection{Bilinear locally weak $L^{\vec{r}}$-bounded }
		Given a bilinear operator $T$ and $1\leqslant r_1, r_2<\infty, \vec{r}=(r_1,r_2), 1/r=1/{r_1}+1/{r_2}$, we say that $T$ is bilinear locally weak $L^{ \vec{r}}$-bounded if there exists a non-increasing function $\varphi_{T,r_1,r_2}:(0,1)\to[0,\infty)$ such that for any $Q, Q^\prime\in\mathcal{Q}$ and $f_1\in L^{r_1}(Q), f_2\in L^{r_2}(Q) $  one has
		$$\Big|\Big\{x\in Q: \big|T(f_1\chi_{Q},f_2\chi_{Q})(x)\big|>\varphi_{T,r_1,r_2}(\lambda)\langle|f_1|\rangle_{r_1,Q}\langle|f_2|\rangle_{r_2,Q}\Big\}\Big|\leqslant\lambda|Q| ,\quad\lambda\in(0,1).$$
		Note that weak $L^{\vec{r}}$-boundedness of $T$ implies the local weak $L^{\vec{r}}$-boundedness of $T$ with
		$$\varphi_{T,r_1,r_2}(\lambda) =\lambda^{-\frac{1}{r}}\|{T}\|_{L^{r_1}(\mathbb{R}^n) \times L^{r_2}(\mathbb{R}^n)\to L^{r,\infty}(\mathbb{R}^n) },\quad\lambda\in(0,1).$$
		\begin{thm}\label{main thm}
			Let $1<s\leqslant \infty,1\leqslant r_1,r_2<\infty, {1}/{r}={1}/{r_1}+{1}/{r_2}$ satisfy $r<s$. Let $k_1 \in \mathbb {N}$ and  $T$ be a bilinear operator. Assume that $T$ and $\mathcal{M}^{\#}_{T,s}$ are locally weak $L^{\vec{r}}$-bounded. Then there exist $C_{k_1,n}>1$ and $\lambda_{k_1,n} <1$ such that for any $f_1,f_2,g\in L^{\infty}_{\rm{c}}(\mathbb{R}^n)$ and $b\in L^{1}_{\rm{loc}}(\mathbb{R}^n)$, there is a ${1}/{(2\cdot3^n)}$-sparse collection of cubes $\mathcal{S}$ such that 
			\begin{align}\label{key formular6}
				\begin{split}
				\int_{\mathbb{R}^n}\big|\mathcal{C}^{k_1}_{b}(T)(f_1,f_2)\big|\cdot |g| \leqslant &\, C_1\Big(\sum_{Q\in \mathcal{S}} \big\langle {|b-\langle{b}\rangle_Q|^{k_1}|f_1|}\big\rangle_{r_1,Q}\big\langle |f_2|\big\rangle_{r_2,Q}\big\langle |g|\big\rangle_{s^{'},Q}|Q|\\
				&\quad +\sum_{Q\in \mathcal{S}} \langle |f_1|\big\rangle_{r_1,Q}\big\langle |f_2|\big\rangle_{r_2,Q}\big\langle {|b-\langle{b}\rangle_Q|^{k_1}|g|}\big\rangle_{s^{\prime},Q}\big|Q|\Big), 
				\end{split}
			\end{align}
			where
			\begin{equation}
				C_1:= C_{k_1,n}\big( \varphi _{T,r_1,r_2}(\lambda_{k_1,n})+\varphi _{\mathcal{M}^{\#}_{T,s},r_1,r_2}(\lambda_{k_1,n})\big)\label{key formular5}.
			\end{equation}
		\end{thm}
		Theorem \ref{main thm} is an immediate consequence of following Lemma \ref{main lem} and Lemma \ref{lem3}, whose proof are essentially the same as in \cite[Theorem 3.3]{MR4718669} and \cite[Lemma 3.4]{MR4718669}, respectively.
		\begin{lem}\label{main lem}
			Under the assumptions of Theorem \ref{main thm}  we have $$\int_{\mathbb{R}^n}\big|\mathcal{C}^{k_1}_{b}(T)(f_1,f_2)\big|\cdot |g|\leqslant C\sum_{k=0}^{k_1}	\Big(\sum_{Q\in \mathcal{S}} \big\langle {|b-\langle{b}\rangle_Q|^{k_1-k}|f_1|}\big\rangle_{r_1,Q}\big\langle |f_2|\big\rangle_{r_2,Q}\big\langle |b-\langle{b}\rangle|^{k}|g|\big\rangle_{s^{'},Q}|Q|\Big),$$
			where $C$ is given by (\ref{key formular5}).
		\end{lem} 
		\begin{lem} \label{lem3}
			Let $1\leqslant r_1, s^\prime < \infty $ and $k_1\in  \mathbb{N}$ . Let $f,g\in  L^{\infty}_{\rm{c}}(\mathbb{R}^n)$ and $b\in L^{1}_{\rm{loc}}(\mathbb{R}^n)$. Fix a cube $Q\in \mathcal{Q}$ and for $0\leqslant k \leqslant k_1$ define
			$$c_k:=\big\langle {|b-\langle{b}\rangle_Q|^{k_1-k}|f|}\big\rangle_{r_1,Q}\big\langle |b-\langle{b}\rangle|^{k}|g|\big\rangle_{s^\prime,Q},$$
			then we have $c_k\leqslant c_0+c_{k_1}$.
		\end{lem}
		%\begin{rem}
		%	Under the assumptions of Theorem \ref{main thm} %and using "three lattice theorem", we know that %there exist $3^n$ dyadic lattices %$\mathcal{D}_j$ so that for any $f_1,f_2,g\in %L_c^\infty(\mathbb{R}^n)$ and $b\in %L_{loc}^1(\mathbb{R}^n)$, there exist sparse %families $S_j\subset \mathcal{D}_j$,
		% 		\begin{align}
			% 			\int_{\mathbb{R}^n}|\mathcal{C}^{k_1}_{b_1}(T)(f_1,f_2)|\cdot |g|&\leqslant C\sum_{j=0}^{3^n}	(\sum_{Q\in \mathcal{S}_j} \langle {|b_1-\langle{b_1}\rangle_Q|^{k_1}|f_1|}\rangle_{r_1,Q}\langle |f_2|\rangle_{r_2,Q}\langle|g|\rangle_{s^{'},Q}|Q|\nonumber\\
			%			&\quad +\sum_{Q\in \mathcal{S}_j} \langle |f_1|\rangle_{r_1,Q}\langle |f_2|\rangle_{r_2,Q}\langle |b_1-\langle{b_1}\rangle|^{k_1}|g|\rangle_{s^{\prime},Q}|Q|.\nonumber
			%	\end{align}
		%	\end{rem}%
	
	\section{Weighted estimates of the sparse form}\label{sec:we}   
	 This section is devoted to establishing a weighted estimate for fractional sparse forms, which is the key to proving Theorem \ref{thm 10}.  
	\begin{thm}\label{mmmain thm}
		Let $1\leqslant r_1 <p_1\leqslant q_1< \infty, 1\leqslant r_2 <p_2 <\infty, 1< p\leqslant q<s\leqslant \infty$ with  ${1}/{p}= {1}/{p_1}+{1}/{p_2}, {1}/{q}={1}/{q_1}+{1}/{p_2}$ and $\beta_1:=r_1(1/{p_1}-1/{q_1})$. Assume that  $ (\lambda_1, \omega_2 )\in A_{\vec{q},\vec{r}}$, where $\vec{q}=(q_1,p_2) , \vec{r}=(r_1, r_2 ,s^\prime)$. For any sparse family cubes $\mathcal{S} \subset \mathscr{D}, f_1\in L^{p_1}(\lambda_1^{p_1})$, $f_2\in L^{p_2}(\omega_2^{p_2})$ and $g\in L^{q^\prime}\big((\lambda_1\omega_2)^{-q^\prime}\big)$ we have
		\begin{equation}\label{key formula2}
			\begin{split}
			\sum_{Q\in \mathcal{S}}\big\langle |f_1| \big\rangle _{\frac{r_1}{1+\beta_1},Q}\big\langle |f_2| \big\rangle_{r_2,Q}\big\langle |g|\big\rangle_{s^\prime,Q}|Q|^{1+\frac{\beta_1}{r_1}}
		\lesssim \big\| f_1\lambda_1\big\|_{L^{p_1}}\big\|f_2\omega_2\big\|_{L^{p_2}}\big\|g(\lambda_1\omega_2)^{-1}\big\|_{L^{q^\prime}},
			\end{split}
		\end{equation}
	\end{thm}
	\begin{proof}
		We follow the idea  in \cite{MR3591468}.
	Denote
		$$\mu_1=\lambda_1^{\frac{r_1q_1}{r_1-q_1}},\quad \mu_2=\omega_2^{\frac{r_2p_2}{r_2-p_2}},\quad 
		\mu_3=(\lambda_1\omega_2)^\frac{sq}{s-q},$$
		then (\ref{key formula2}) can be rewritten as
		\begin{align}
		\sum_{Q\in \mathcal{S}}	&\big\langle |f_1| \big\rangle _{\frac{r_1}{1+\beta_1},Q}\big\langle |f_2| \big\rangle_{r_2,Q}\big\langle |g|\big\rangle_{s^\prime,Q}|Q|^{1+\frac{\beta_1}{r_1}}\nonumber\\
		& \quad\quad\quad\lesssim\big\| |f_1| ^{\frac{r_1}{1+\beta_1}}\mu_1^{-1}\big\|^{\frac{1+\beta_1}{r_1}}_{L^{\frac{p_1(1+\beta_1)}{r_1}}(\mu_1)}\big\||f_2|^{r_2}\mu_2^{-1}\big\|^{\frac{1}{r_2}}_{L^{\frac{p_2}{r_2}}(\mu_2)}\big\||g|^{s^\prime}\mu_3^{-1}\big\|^{\frac{1}{s^\prime}}_{L^{\frac{q^\prime}{s^\prime}}(\mu_3)}.\nonumber
		\end{align}
		Then a simple change of variable reduces the problem to
		\begin{align}\label{key formula4}
			\begin{split}
			\sum_{Q\in \mathcal{S}}\big\langle |f_1| \big\rangle^{\mu_1}_{\frac{r_1}{{1+\beta_1}},Q}&\big\langle |f_2|\big\rangle^{\mu_2}_{r_2,Q}\big\langle|g|\big\rangle^{\mu_3}_{s^\prime,Q}|Q|^{1+\frac{\beta_1}{r_1}}\langle\mu_1\rangle_Q^{\frac{1+\beta_1}{r_1}}\langle \mu_2 \rangle_Q^{\frac{1}{r_2}}\langle \mu_3\rangle_Q^{\frac{1}{s^\prime}}\\
			& \lesssim \big\| |f_1| ^{\frac{r_1}{1+\beta_1}}\big\|^{\frac{1+\beta_1}{r_1}}_{L^{\frac{p_1(1+\beta_1)}{r_1}}(\mu_1)}\big\||f_2|^{r_2}\big\|^{\frac{1}{r_2}}_{L^{\frac{p_2}{r_2}}(\mu_2)}\big\||g|^{s^\prime}\big\|^{\frac{1}{s^\prime}}_{L^{\frac{q^\prime}{s^\prime}}(\mu_3)}\\
			& = \big\| f_1 \big\|_{L^{p_1}(\mu_1)}\big\|f_2\big\|_{L^{p_2}(\mu_2)}\big\|g\big\|_{L^{q^\prime}(\mu_3)}.
		 \end{split}
		\end{align}
		Then we claim that (\ref{key formula4}) is equivalent to 
		\begin{align}\label{key formula5}
			\begin{split}
			\sum_{Q\in \mathcal{S}}&\big\langle |f_1| \big\rangle^{\mu_1} _{Q}\big\langle |f_2| \big\rangle^{\mu_2}_{Q}\big\langle |g|\big\rangle^{\mu_3}_{Q}|Q|^{1+\frac{\beta_1}{r_1}}\langle\mu_1\rangle_Q^{\frac{1+\beta_1}{r_1}}\langle \mu_2 \rangle_Q^{\frac{1}{r_2}}\langle \mu_3\rangle_Q^{\frac{1}{s^\prime}}\\
			&\quad\quad\quad\lesssim  \| f_1 \|_{L^{p_1}(\mu_1)}\|f_2\|_{L^{p_2}(\mu_2)}\|g\|_{L^{q^\prime}(\mu_3)}.
			\end{split}
		\end{align}
		We prove this equivalence via the following two cases. 
		
		\textbf{The case $ {r_1}\geq 1+\beta_1$.}
		 If (\ref{key formula4}) holds,  (\ref{key formula5}) holds by H\"older's inequality.
		Conversely, when (\ref{key formula5}) holds, then
		\begin{align}
			\sum_{Q\in \mathcal{S}}&\big\langle |f_1| \big\rangle^{\mu_1}_{\frac{r_1}{{1+\beta_1}},Q}\big\langle |f_2|\big\rangle^{\mu_2}_{r_2,Q}\big\langle|g|\big\rangle^{\mu_3}_{s^\prime,Q}|Q|^{1+\frac{\beta_1}{r_1}}\langle\mu_1\rangle_Q^{\frac{1+\beta_1}{r_1}}\langle \mu_2 \rangle_Q^{\frac{1}{r_2}}\langle \mu_3\rangle_Q^{\frac{1}{s^\prime}}\nonumber\\
			&\lesssim\sum_{Q\in \mathcal{S}}\big\langle M^{\mathcal{S}}_{\frac{r_1}{1+\beta_1},\mu_1}(f_1)\big\rangle_Q^{\mu_1}\big\langle M^{\mathcal{S}}_{r_2,\mu_2}(f_2)\big\rangle_Q^{\mu_2}\big\langle M^{\mathcal{S}}_{s^\prime,\mu_3}(g)\big\rangle_Q^{\mu_3}|Q|^{1+\frac{\beta_1}{r_1}}\langle\mu_1\rangle_Q^{\frac{1+\beta_1}{r_1}}\langle \mu_2 \rangle_Q^{\frac{1}{r_2}}\langle \mu_3\rangle_Q^{\frac{1}{s^\prime}}\nonumber\\
		    &\lesssim \big\|M^{\mathcal{S}}_{\frac{r_1}{1+\beta_1},\mu_1}(f_1)\big\|_{L^{p_1}(\mu_1)}\big\|M^{\mathcal{S}}_{r_2,\mu_2}(f_2)\big\|_{L^{p_2}(\mu_2)}\big\|M^{\mathcal{S}}_{s^\prime,\mu_3}(g)\big\|_{L^{q^\prime}(\mu_3)}\nonumber\\
			&\lesssim \|f_1\|_{L^{p_1}(\mu_1)}\|f_2\|_{L^{p_2}(\mu_2)}\|g\|_{L^{q^\prime}(\mu_3)},\nonumber
		\end{align}
		where
		$$M^S_{p,\mu}(f):=\sup \limits_{Q\in \mathcal{S}}\big(\langle|f|^p\rangle^\mu_Q \big)^{\frac{1}{p}}.$$
		The last inequality holds since $$\frac{(1+\beta_1)p_1}{r_1}=\frac{p_1}{r_1}+1-\frac{p_1}{q_1}>1.$$

		\textbf{The case  $ {r_1}<{1+\beta_1} $.} For necessity, we use H\"older's inequality in the first step
		\begin{align}
			\sum_{Q\in \mathcal{S}}&\big\langle |f_1| \big\rangle^{\mu_1}_{\frac{r_1}{{1+\beta_1}},Q}\big\langle |f_2|\big\rangle^{\mu_2}_{r_2,Q}\big\langle|g|\big\rangle^{\mu_3}_{s^\prime,Q}|Q|^{1+\frac{\beta_1}{r_1}}\langle\mu_1\rangle_Q^{\frac{1+\beta_1}{r_1}}\langle \mu_2 \rangle_Q^{\frac{1}{r_2}}\langle \mu_3\rangle_Q^{\frac{1}{s^\prime}}\nonumber\\
		    &\lesssim\sum_{Q\in \mathcal{S}}\big\langle |f_1|\big\rangle_Q^{\mu_1}\big\langle M^{\mathcal{S}}_{r_2,\mu_2}(f_2)\big\rangle_Q^{\mu_2}\big\langle M^{\mathcal{S}}_{s^\prime,\mu_3}(g)\big\rangle_Q^{\mu_3}|Q|^{1+\frac{\beta_1}{r_1}}\langle\mu_1\rangle_Q^{\frac{1+\beta_1}{r_1}}\langle \mu_2 \rangle_Q^{\frac{1}{r_2}}\langle \mu_3\rangle_Q^{\frac{1}{s^\prime}}\nonumber\\
			&\lesssim \|f_1\|_{L^{p_1}(\mu_1)}\big\|M^{\mathcal{S}}_{r_2,\mu_2}(f_2)\big\|_{L^{p_2}(\mu_2)}\big\|M^{\mathcal{S}}_{s^\prime,\mu_3}(g)\big\|_{L^{q^\prime}(\mu_3)}\nonumber\\
			&\lesssim \|f_1\|_{L^{p_1}(\mu_1)}\|f_2\|_{L^{p_2}(\mu_2)}\|g\|_{L^{q^\prime}(\mu_3)}.\nonumber
		\end{align}
		Conversely, for sufficiency
		\begin{align}
			\sum_{Q\in \mathcal{S}}&\big\langle |f_1| \big\rangle^{\mu_1} _{Q}\big\langle |f_2| \big\rangle^{\mu_2}_{Q}\big\langle |g|\big\rangle^{\mu_3}_{Q}|Q|^{1+\frac{\beta_1}{r_1}}\langle\mu_1\rangle_Q^{\frac{1+\beta_1}{r_1}}\langle \mu_2 \rangle_Q^{\frac{1}{r_2}}\langle \mu_3\rangle_Q^{\frac{1}{s^\prime}}\nonumber\\
			&\lesssim\sum_{Q\in \mathcal{S}}\big\langle M^{\mathcal{S}}_{1,\mu_1}(f_1) \big\rangle^{\mu_1}_{\frac{r_1}{{1+\beta_1}},Q}\big\langle |f_2|\big\rangle^{\mu_2}_{r_2,Q}\big\langle|g|\big\rangle^{\mu_3}_{s^\prime,Q}|Q|^{1+\frac{\beta_1}{r_1}}\langle\mu_1\rangle_Q^{\frac{1+\beta_1}{r_1}}\langle \mu_2 \rangle_Q^{\frac{1}{r_2}}\langle \mu_3\rangle_Q^{\frac{1}{s^\prime}}\nonumber\\
			&\lesssim  \big\|M^{\mathcal{S}}_{1,\mu_1}(f_1)\big\|_{L^{p_1}(\mu_1)}\|f_2\|_{L^{p_2}(\mu_2)}\|g\|_{L^{q^\prime}(\mu_3)}\nonumber\\
			&\lesssim \|f_1\|_{L^{p_1}(\mu_1)}\|f_2\|_{L^{p_2}(\mu_2)}\|g\|_{L^{q^\prime}(\mu_3)}.\nonumber
		\end{align}
	    Since $ (\lambda_1, \omega_2 )\in A_{\vec{q},\vec{r}}$ and  $\beta_1={r_1}/{p_1}-{r_1}/{q_1}$, we have
     	\begin{align}
    	\sum_{Q\in \mathcal{S}}&\big\langle |f_1| \big\rangle^{\mu_1} _{Q}\big\langle |f_2| \big\rangle^{\mu_2}_{Q}\big\langle |g|\big\rangle^{\mu_3}_{Q}|Q|^{1+\frac{\beta_1}{r_1}}\langle\mu_1\rangle_Q^{\frac{1+\beta_1}{r_1}}\langle \mu_2 \rangle_Q^{\frac{1}{r_2}}\langle \mu_3\rangle_Q^{\frac{1}{s^\prime}}\nonumber\\
    	&\lesssim \sum_{Q\in \mathcal{S}}\big\langle |f_1| \big\rangle^{\mu_1} _{Q}\big\langle |f_2| \big\rangle^{\mu_2}_{Q}\big\langle |g|\big\rangle^{\mu_3}_{Q}\mu_1(Q)^{\frac{1}{p_1}}\mu_2(Q)^{\frac{1}{p_2}}\mu_3(Q)^{1-\frac{1}{q}}\nonumber\\
    	&\leq  \Big(\sum_{Q\in \mathcal{S}}\big(\langle |f_1| \rangle^{\mu_1}_{Q}\big)^{p_1}\mu_1(Q)\Big)^{\frac{1}{p_1}}\Big(\sum_{Q\in \mathcal{S}}\big(\langle |f_2| \rangle^{\mu_2}_{Q}\big)^{p_2}\mu_2(Q)\Big)^{\frac{1}{p_2}}\nonumber\\
    	&\quad\quad\quad\quad\times\Big(\sum_{Q\in \mathcal{S}}\big(\langle |g| \rangle^{\mu_1}_{Q}\big)^{p^\prime}\mu_3(Q)^{\frac{p^\prime}{q^\prime}}\Big)^{\frac{1}{p^\prime}}\nonumber,
	\end{align}
	Since ${1}/{p_1}+{1}/{p_2}+{1}/{p^\prime}=1$, we have used H\"older's inequality in the last step. For ${p^\prime}/{q^\prime}\geq1$, we obtain
	$$\Big(\sum_{Q\in \mathcal{S}}\big(\langle |g| \rangle^{\mu_1}_{Q}\big)^{p^\prime}\mu_3(Q)^{\frac{p^\prime}{q^\prime}}\Big)^{\frac{1}{p^\prime}}\leq \Big(\sum_{Q\in \mathcal{S}}\big(\langle |g| \rangle^{\mu_1}_{Q}\big)^{q^\prime}\mu_3(Q)\Big)^{\frac{1}{q^\prime}}\lesssim \Big(\sum_{Q\in \mathcal{S}}\int_{E_{Q}}\big(\mathcal{M}_{1,\mu_3}(g)\big)^{q^\prime}\mu_3\Big)^{\frac{1}{q^\prime}}.$$
     Now, we arrive at
	\begin{align}
    	\sum_{Q\in \mathcal{S}}	&\big\langle |f_1| \big\rangle^{\mu_1} _{Q}\big\langle |f_2| \big\rangle^{\mu_2}_{Q}\big\langle |g|\big\rangle^{\mu_3}_{Q}|Q|^{1+\frac{\beta_1}{r_1}}\langle\mu_1\rangle_Q^{\frac{1+\beta_1}{r_1}}\langle \mu_2 \rangle_Q^{\frac{1}{r_2}}\langle \mu_3\rangle_Q^{\frac{1}{s^\prime}}\nonumber\\
	    &\lesssim \|\mathcal{M}_{1,\mu_1}(f_1)\|_{L^{p_1}(\mu_1)}\|\mathcal{M}_{1,\mu_2}(f_2)\|_{L^{p_2}(\mu_2)}\|\mathcal{M}_{1,\mu_3}(g)\|_{L^{q^\prime}(\mu_3)}\nonumber\\
		&\leq \|f_1\|_{L^{p_1}(\mu_1)}\|f_2\|_{L^{p_2}(\mu_2)}\|g\|_{L^{q^\prime}(\mu_3)}\nonumber.
	\end{align}
	\end{proof}
	\section{proof of the main theorem}\label{sec:pro}
	\begin{defn}
		Given a weight $\nu$ and $\alpha \geq 0$ , we say
		$b\in BMO_\nu^\alpha$ if
		$$\|b\|_{BMO_\nu^\alpha}:=\sup\limits_{Q\in\mathcal{Q}}\frac{1}{\nu(Q)^{1+\alpha}}\int_{Q}\big|b-\langle b \rangle_{Q}\big| < \infty.$$
		We omit $\alpha $ when $\alpha=0$. Moreover , we define the weighted sharp maximal function as following
		$$M_\nu ^\#(b)=\sup\limits_{Q\in \mathcal{Q}}\frac{\chi_{Q}}{\nu(Q)}\int_{Q}\big|b-\langle b \rangle _Q\big|=:\sup\limits_{Q\in \mathcal{Q}}\Omega_{\nu}(b,Q)\chi_Q.$$
		Note that $\|b\|_{BMO_\nu}=\|M_\nu ^\#(b)\|_{L^{\infty}(\mathbb{R}^n)}.$
	\end{defn}
	\begin{defn} 
		Let  $1< p_1,p_2,q_1< \infty, \vec{p}=(p_1,p_2), \vec{q}=(q_1,p_2)$ with ${1}/{p}= {1}/{p_1}+{1}/{p_2}, {1}/{q}={1}/{q_1}+{1}/{p_2}$.  Assume that $(\omega_1,\omega_2)\in A_{\vec{p}}, (\lambda_1,\omega_2) \in A_{\vec{q}}$ and $\alpha_1:={1}/{(p_1k_1)}-{1}/{(q_1k_1)}$. For $k_1\in \mathbb{N} $, we define the Bloom weight
		$\nu_1$ via
		\begin{equation}
			\nu_1^{1+\alpha_1}:=\omega_1^{\frac{1}{k_1}}\lambda_1^{-\frac{1}{k_1}}.\label{ fomulater1}
		\end{equation}
	\end{defn}
	To our knowledge, this is the first time to define the  off-diagonal bilinear Bloom weight. Therefore, we first show the necessity. Specially, we will prove Theorem \ref{thm 2}. Before we start the proof, we record the following proposition, which is very useful in the proof.
	\begin{lem}\emph{(}\cite[Proposition 2.14]{MR4453692}\emph{)}\label{key lem3}
		Let $\sigma\in A_{\infty}$. Then there holds
		\begin{equation}
			|b-b^{\sigma}_{Q_0}|\chi_{Q_0} \lesssim \sum_{Q\in \mathcal{S}(Q_0)} \big\langle |b-b_{Q}^{\sigma}|\big\rangle_{Q}^{\sigma}\chi_Q,\nonumber
		\end{equation}
		where $\mathcal{S}(Q_0) $ is a $\frac{1}{2}$-sparse collection with all respect to $\sigma$ and with all elements contained in $Q_0$.
	\end{lem}
	\begin{proof}[Proof of Theorem \ref{thm 2}, (1)] 
     We follow the strategy in \cite{MR4453692}. By (\ref{formula 5}), for any $A\geq3$ and cube $Q$, we can find $\tilde{Q}$ satisfying $\ell(Q)= \ell(\tilde{Q})$ and ${\rm{dist}}(Q,\tilde{Q})\simeq A\ell(Q)$, and there exists some $\sigma\in \mathbb{C}$ with $|\sigma|=1$ such that for all $x\in \tilde{Q}$ and $y_1, y_2 \in Q$, we have
	\begin{equation}
		\operatorname{Re} \big(\sigma K(x, y_1, y_2)\big) \gtrsim_{A} \frac{1}{|Q|^{2}}.\label{key fomular5}
	\end{equation}
	Indeed, for a fixed cube $Q$ and any point $y_0\in Q$, we can find a point $x_0\in (AQ)^{c}$, such that 
	\begin{equation}
		\frac{1}{c_0(Al(Q))^{2d}}\lesssim|K(x_0,y_{0},y_{0})|\lesssim \frac{c_K}{(|x_0-y_{0}|+|x_0-y_{0}|)^{2d}}.\nonumber
	\end{equation}
	Hence, $Al(Q)\leq|x_0-y_0|\leq (c_0c_K)^{1/{(2d)}}Al(Q)$.
	Now, we use $\tilde{Q}$ to denote the cube with center  $x_0$ and side length $l(Q)$ and let $\sigma={\overline{K(x_0,y_0,y_0)}}/{|K(x_0,y_0,y_0)|}$, for any $x\in \tilde{Q}$ and $y_1, y_2 \in Q$
	\begin{align}
		\operatorname{Re} \big(\sigma K(x, y_1, y_2)\big)\geq |K(x_0,y_{0},y_{0})|-|K(x,y_1,y_2)-K(x_0,y_{0},y_{0})|,\nonumber
	\end{align}
	then
	\begin{align}
		|K(x,y_1,&y_2)-K(x_0,y_{0},y_{0})|\nonumber\\
		\leq&|K(x,y_1,y_2)-K(x_0,y_1,y_2)|+|K(x_0,y_1,y_2)-K(x_0,y_{0},y_2)|\nonumber\\
		&+|K(x_0,y_{0},y_2)-K(x_0,y_{0},y_{0})|\nonumber\\
		\lesssim& \frac{1}{(|x-y_1|+|x-y_2|)^{2d}}\omega\Big(\frac{|x-x_0|}{|x-y_1|+|x-y_2|}\Big)\nonumber\\
		&+\frac{1}{(|x_0-y_1|+|x_0-y_2|)^{2d}}\omega\Big(\frac{|y_1-y_{0}|}{|x_0-y_1|+|x_0-y_2|}\Big)\nonumber\\
		&+\frac{1}{(|x_0-y_{0}|+|x_0-y_2|)^{2d}}\omega\Big(\frac{|y_2-y_{0}|}{|x_0-y_{0}|+|x_0-y_2|}\Big)\nonumber\\
		\lesssim& \frac{2}{(Al(Q)-l(Q))^{2d}}\omega\Big(\frac{1}{A-1}\Big)+\frac{1}{(Al(Q))^{2d}}\omega\Big(\frac{1}{A}\Big)\nonumber\\
		=& \frac{1}{(Al(Q))^{2d}}\Big(\frac{2}{(1-A^{-1})^{2d}}\omega\Big(\frac{1}{A-1}\Big)+\omega\Big(\frac{1}{A}\Big)\Big)\nonumber\\
		\leq&\frac{\varepsilon_A}{(Al(Q))^{2d}},\nonumber
	\end{align}
	where $\varepsilon_A\to 0 $ as $A\to \infty $ by the condition that $\omega(t)\to0$ as $t\to0$. Hence, we obtain (\ref{key fomular5}).

		Set $v=(\lambda_1\omega_2)^q,  \sigma_1=\omega_1^{-p_1^\prime} ,\sigma_2=\omega_2^{-p_2^\prime}$. 
		For arbitrary $\beta \in \mathbb{R}$ and $x\in \tilde{Q}\cap \{b\geq\beta\}$, we have
		\begin{align}
			\frac{\sigma_2(Q)}{|Q|^2}\int_{Q}(\beta-b)_{+}^{k_1}\sigma_1\leq&\frac{1}{|Q|^2}\int_{Q\cap \{b\leqslant \beta\}}\int_{Q} \big(b(x)-b(y_1)\big)^{k_1} \sigma_1(y_1) \sigma_2(y_2) dy_{2}dy_{1}\nonumber\\
			\lesssim& \operatorname{Re}\Big(\sigma\int_{Q\cap \{b\leqslant \beta\}}\int_{Q}\big(b(x)-b(y_1)\big)^{k_1}K(x,y_1, y_2)\sigma_1(y_1)\sigma_2(y_2)dy_{2}dy_{1}\Big)\nonumber.
		\end{align}
		Let $\beta$ be a median of $b$ in $\tilde{Q}$, namely
		$$\min\big(|\tilde{Q}\cap\{b\geq \beta\}|,|\tilde{Q}\cap\{b\leq \beta\}|\big)\geq\frac{1}{2}|\tilde{Q}|=\frac{1}{2}|Q|,$$
		since $v\in A_{\infty}$ is doubling, we have
		$$\big|\tilde{Q}\cap\{b\geq \beta\}\big| \thicksim |\tilde{Q}| \Rightarrow v\big(\tilde{Q}\cap\{b\geq \beta\}\big) \thicksim v(\tilde{Q}) \thicksim_{A} v(Q),$$
		As a result,
		\begin{align}
			&v(Q)^{\frac{1}{q}}\frac{\sigma_2(Q)}{|Q|^2}\int_{Q}(\beta-b)_{+}^{k_1}\sigma_1\nonumber\\
			&\lesssim \Big\| \operatorname{Re}\sigma\int_{Q\cap \{b\leqslant \beta\}}\int_{Q}(b(x)-b(y_1))^{k_1}K(x,y_1, y_2)\sigma_1(y_1)\sigma_2(y_2)dy_{2}dy_{1} \Big\| _{L^{q,\infty}(\tilde{Q} \cap \{b\geq \beta \};v)}\nonumber\\
			&\leq\|\chi_{\tilde{Q}}\mathcal{C}_{b}^{k_1}(T)(\chi_{Q\cap \{b\leqslant \beta\}}\sigma_1,\chi_{Q}\sigma_2)\|_{L^{q,\infty}(\nu)}\nonumber\\
			&\lesssim \mathcal{N} \sigma_1(Q)^{\frac{1}{p_1}}\sigma_2(Q)^{\frac{1}{p_2}},\nonumber
		\end{align}
		where $\mathcal{N}= \|C_{b}^{k_1}(T)\|_{L^{p_1}(\omega_1^{p_1})\times L^{p_2}(\omega_2^{p_2})\to L^{q,\infty}((\lambda_1\omega_2)^{q})}$.
		In conclusion, we obtain 
		\begin{equation}\label{eq:511}
		\frac{\langle v \rangle_Q^{\frac{1}{q}} \langle \sigma_2 \rangle_Q^{\frac{1}{p_2^\prime}}}{\sigma_1(Q)^{\frac{1}{p_1}}}\Big({\frac{1}{|Q|}}\Big)^{\frac{1}{q_1^\prime}}\int_{Q}(\beta-b)_{+}^{k_1}\sigma_1\lesssim \mathcal{N}.
		\end{equation}
		Let $\eta_1=\lambda_1^{-q_1^\prime}$,  using H\"older's inequality , 
		$$\langle v\rangle_Q^{\frac{1}{q}}\langle \eta_1\rangle_Q^{\frac{1}{q_1^\prime}}\langle \sigma_2\rangle_Q^{\frac{1}{p_2^\prime}} \geq 1,$$
	plug this into \eqref{eq:511} we have 
		$$\int_{Q}(\beta-b)_{+}^{k_1}\sigma_1\lesssim \mathcal{N} \eta_1(Q)^{\frac{1}{q_1^\prime}}\sigma_1(Q)^{\frac{1}{p_1}}.$$
	On the other hand, since 
		$$\frac{1}{|Q|}\int_{Q}(\beta-b)_{+}\sigma_1^{\frac{1}{k_1}}\leq \Big(\frac{1}{|Q|}\int_{Q}(\beta-b)_{+}^{k_1}\sigma_1\Big)^{\frac{1}{k_1}},$$
		We finally  arrive at
		$$|Q|^{\frac{1}{k_1}-{\frac{1}{q_1^\prime k_1}}-\frac{1}{p_1k_1}-1}\int_{Q}(\beta-b)_{+}\sigma_1^{\frac{1}{k_1}}\lesssim \mathcal{N}^{\frac{1}{k_1}}\langle \eta_1 \rangle _Q^{\frac{1}{q_1^\prime k_1}} \langle \sigma_1 \rangle_Q^{\frac{1}{p_1k_1}}.$$
		Following similar arguments as above, we also have
		$$|Q|^{\frac{1}{k_1}-{\frac{1}{q_1^\prime k_1}}-\frac{1}{p_1k_1}-1}\int_{Q}(b-\beta)_{+}\sigma_1^{\frac{1}{k_1}}\lesssim \mathcal{N}^{\frac{1}{k_1}}\langle \eta_1 \rangle _Q^{\frac{1}{q_1^\prime k_1}} \langle \sigma_1 \rangle_Q^{\frac{1}{p_1k_1}},$$
		hence
		$$\int_{Q}\big|b-\beta\big|\sigma_1^{\frac{1}{k_1}}\lesssim \mathcal{N}^{\frac{1}{k_1}}|Q|^{1+\alpha_1}\langle \eta_1 \rangle _Q^{\frac{1}{q_1^\prime k_1}} \langle \sigma_1 \rangle_Q^{\frac{1}{p_1k_1}} .$$
	Observe that 
		\begin{align}\label{formula 11}
			\begin{split}
			\frac{\int_{Q}|b-\langle b\rangle_{Q}^{\sigma_1^{1/{k_1}}}|\sigma_1^{\frac{1}{k_1}}}{|Q|^{1+\alpha_1}\langle \eta_1 \rangle_Q^{\frac{1}{q_1^\prime k_1}} \langle \sigma_1 \rangle_Q^{\frac{1}{p_1k_1}}}
			\leq \frac{\Big(\int_{Q}|b-\beta|\sigma_1^{\frac{1}{k_1}}+\sigma_1^{\frac{1}{k_1}}(Q)|\beta-\langle b \rangle_Q^{\sigma_1^{1/{k_1}}}|\Big)}{|Q|^{1+\alpha_1}\langle \eta_1 \rangle _Q^{\frac{1}{q_1^\prime k_1}} \langle \sigma_1 \rangle_Q^{\frac{1}{p_1k_1}}} 
			\lesssim\mathcal{N}^{\frac{1}{k_1}}.
		\end{split}
		\end{align}  
		Fix a cube $Q_0$, similar to the triangle inequality above, we have
		$$	\frac{1}{\nu_1(Q_0)^{1+\alpha_1}}\int_{Q_0}\big|b-\langle b \rangle_{Q_0}\big|\leq \frac{2}{{\nu_1}(Q_0)^{1+\alpha_1}}\int_{Q_0}\big|b-\langle b\rangle_{Q_0}^{\sigma_1^{1/{k_1}}}\big|.$$
		Since $\sigma_1\in A_{2p_1^\prime}$, it is easy to see  $\sigma_1^{\frac{1}{k_1}}\in A_{2p_1^\prime}$ by H\"older's inequality. Then using Lemma \ref{key lem3}, we get
		\begin{equation}
			\frac{2}{{\nu_1}(Q_0)^{1+\alpha_1}}\int_{Q_0}\abs{b-\langle b\rangle_{Q_0}^{\sigma_1^{1/{k_1}}}}\lesssim\frac{2}{{\nu_1}(Q_0)^{1+\alpha_1}}\sum_{Q\in \mathcal{S}(Q_0)} \frac{|Q|}{\sigma^{\frac{1}{k_1}}(Q)}\int_{Q}\big|b-\langle b \rangle_Q^{\sigma_1^{1/{k_1}}}\big|\sigma_1^{\frac{1}{k_1}}.\nonumber
		\end{equation}
		According to the inequality (\ref{formula 11}), we arrive at
		\begin{equation}
			\frac{2}{{\nu_1}(Q_0)^{1+\alpha_1}}\int_{Q_0}\abs{b-\langle b\rangle_{Q_0}^{\sigma_1^{1/{k_1}}}}\lesssim \mathcal{N} \frac{2}{{\nu_1}(Q_0)^{1+\alpha_1}}\sum_{Q\in \mathcal{S}(Q_0)}\frac{|Q|^{1+\alpha_1}\big\langle \eta_1 \rangle _Q^{\frac{1}{q_1^\prime k_1}} \langle \sigma_1 \rangle_Q^{\frac{1}{p_1k_1}}}{\langle \sigma_1^{\frac{1}{k_1}} \rangle_Q}.\nonumber
		\end{equation}
		Since $\sigma_1\in A_{\infty}$, by  Lemma \ref{key lem2} we have 
		\[
	\langle \sigma_1^{\frac 1{k_1}}\rangle_Q^{-1}	\langle \sigma_1\rangle_Q^{\frac 1{p_1k_1}}\lesssim  \langle \sigma_1 \rangle_Q^{-1/{k_1}}\langle \sigma_1 \rangle_Q^{\frac 1{p_1k_1}}= \langle \sigma_1 \rangle_Q^{-1/{p_1'k_1}}.
		\]
		Then since $\eta_1\in A_\infty$ and 
		\[
		\nu_1 ^{1+\alpha_1}\sigma_1^{\frac{1}{p_1' k_1}}= \eta_1^{\frac{1}{q_1'k_1}},
		\]
		by  Lemma \ref{key lem2} we have
		\[
		\langle \eta_1\rangle_Q^{\frac 1{q_1'k_1}}\lesssim  \langle \eta_1^{\frac{1}{(1+\alpha_1)q_1'k_1}}\rangle_Q^{1+\alpha_1} \lesssim \langle \nu_1\rangle_Q^{1+\alpha_1} \langle \sigma_1^{\frac 1{(1+\alpha_1)p_1'k_1}}\rangle_Q^{1+\alpha_1}\le 
		\langle \nu_1\rangle_Q^{1+\alpha_1} \langle \sigma_1\rangle_Q^{\frac 1{p_1'k_1}}.
		\]
		Combining the above we get   
		\begin{align}
			\frac{1}{\nu_1(Q_0)^{1+\alpha_1}}\int_{Q_0}\big|b-\langle b \rangle_{Q_0}\big|\lesssim\mathcal{N}^{\frac{1}{k_1}}\frac{2}{{\nu_1}(Q_0)^{1+\alpha_1}}\sum_{Q\in \mathcal{S}(Q_0)} |Q|^{1+\alpha_1}\langle \nu_1\rangle_Q^{1+\alpha_1}.\nonumber
		\end{align}
	Since $\alpha_1\geq0$,
		\begin{align}
		\sum_{Q\in \mathcal{S}(Q_0)}& |Q|^{1+\alpha_1}\langle \nu_1 \rangle_Q^{1+\alpha_1}\leq\Big(\sum_{Q\in \mathcal{S}(Q_0)}\nu_1(Q)\Big)^{1+\alpha_1}\nonumber\\
		&\lesssim\Big(\sum_{Q\in \mathcal{S}(Q_0)}\nu_1(E(Q))\Big)^{1+\alpha_1}\lesssim\nu_1(Q_0)^{1+\alpha_1},\nonumber
		\end{align}
		and we finish the proof.
	\end{proof}
	Next we turn to the second part of Theorem \ref{thm 2}. We begin with a sparse domination principle for the weighted dyadic sharp maximal function. 
	
			\begin{prop}\label{prop}
			Let $f\in L^\infty_{\rm{c}}(\mathbb{R}^n)$. If $\nu\in A_{\infty}$, then the dyadic sharp maximal function $$M_{\nu,d}^{\#}(f):=\sup\limits_{Q\in \mathscr{D}}\frac{1}{\nu(Q)}\big( \int_{Q}|f-\langle f \rangle _Q|\big) \chi_{Q},$$ satisfies 
			\begin{equation}
				M_{\nu,d}^{\#}(f)(x)\lesssim \sum_{Q\in \mathcal{S}} \frac{\chi_{Q}(x)}{\nu(Q)}\int_{Q}\big |f- \langle f \rangle_Q\big|,\nonumber
			\end{equation}
			where $\mathcal{S}\subset \mathscr{D}$ is sparse collection. 
		\end{prop}
		\begin{proof}
			Fixed $C_0>1$ is a constant and chosen later. For $k\in \mathbb{Z}$, we define the set $\Omega_k(x):=\big\{x\in \mathbb{R}^n: M_{\nu,d}^{\#}f(x)>C_0^k\big\}$ and we can select $\{Q_j^k\}$, a mutually disjoint collection of maximal dyadic cubes, such that
			\begin{equation}
				\frac{1}{\nu(Q_j^k)}\int_{Q_j^k}\big|f-\langle f \rangle _{Q_j^k}\big| >C_0^k,\nonumber
			\end{equation}
			then $\Omega_k=\bigcup\limits_{j} Q_j^k$. Let $\mathcal{S}=\bigcup\limits_{j,k}\{Q_j^k\}$, we have
			\begin{align}
				M_{\nu,d}^{\#}f(x) & \lesssim \sum_{k\in \mathbb{Z}}\big[\chi_{\Omega_k}(x)-\chi_{\Omega_{k+1}}(x)\big]C_0^k \nonumber	\\
				& < \sum_{k\in \mathbb{Z}}\chi_{\Omega_k}(x)C_0^k\nonumber\nonumber\\
				&\lesssim \sum_{k\in \mathbb{Z}} \sum_{j} \frac{\chi_{Q_j^k}(x)}{\nu(Q_j^k)} \int_{Q_j^k} \big|f-\langle f \rangle _{Q_j^k}\big|\nonumber\\
				&= \sum_{Q\in \mathcal{S}}\frac{\chi_{Q}(x)}{\nu(Q)} \int_{Q} \big|f-\langle f \rangle _{Q}\big|.\nonumber
			\end{align}
			Now, we will prove that $\mathcal{S}$ is sparse collection. By the maximality, the collection $\mathcal S$ is nested. So we may define $$E(Q_j^k):=Q_j^k\setminus \bigcup_{\substack{Q_l^{k+1}\in \mathcal{S}\\ Q_l^{k+1}\subset Q_j^k}}Q_l^{k+1}.$$ Denote by $(Q_j^k)^{(1)}$  the dyadic  parent of $Q_j^k$, we obtain
			\begin{align}
				\sum_{\substack{l:Q_l^{k+1}\in \mathcal{S}\\ Q_l^{k+1}\subset Q_j^k}} \nu(Q_l^{k+1})&\leq \sum_{l} \frac{1}{C_0^{k+1}}\int_{Q_l^{k+1}}\big|f-\langle f \rangle_{Q_l^{k+1}}\big|\nonumber\\
				&\leq\sum_{l} \frac{1}{C_0^{k+1}}\Big(\int_{Q_l^{k+1}}\big|f-\langle f \rangle_{(Q_j^{k})^{(1)}}\big|+\int_{Q_l^{k+1}}\big|f-\langle f \rangle_{(Q_j^{k})^{(1)}}\big|\Big)\nonumber\\
				&\leq \frac{2C_0^k}{C_0^{k+1}}\nu\big((Q_j^k)^{(1)}\big)
				\leq  \frac{2C_1}{C_0} \nu(Q_j^k),\nonumber
			\end{align}
			where we set  $C_1$ to be the doubling constant of $\nu$. Since $\nu \in A_\infty$, there exist $C_2,\theta>0$ such that for all cubes $Q$ and all measurable subsets $A$ of $Q$ we have \[
			\frac{|A|}{|Q|} \le C_2 \left(\frac{\nu(A)}{\nu(Q)} \right)^{\theta}.
			\] Apply this property to 
			\[
			A= \bigcup_{\substack{Q_l^{k+1}\in \mathcal{S}\\ Q_l^{k+1}\subset Q_j^k}}Q_l^{k+1},
			\]
			 we can choose a sufficiently large $C_0$ so that  $|E(Q_j^k)|\geq \frac{1}{2} |Q_j^k|$.
			We obtain the sparse domination. 
		\end{proof}
		
			\begin{proof}[Proof of Theorem \ref{thm 2}, (2)] 

We first show that under the condition of Theorem \ref{thm 2}, we have the following 
	\begin{align}
		\sum_{i=1}^{N}\big|\big\langle C_{b}^{k_1}(T)(f_{1,i},f_{2,i}),g_i\big\rangle\big|&\lesssim\Big\|\sum_{i=1}^{N}\big\|f_{1,i}\big\|_{\infty}\chi_{Q_i}\Big\|_{L^{p_1}(\omega_1^{p_1})}\Big\|\sum_{i=1}^{N}\big\|f_{2,i}\big\|_{\infty}\chi_{Q_i}\Big\|_{L^{p_2}(\omega_2^{p_2})}\nonumber\\
		&\times \Big\|\sum_{i=1}^{N}\big\|g_i\big\|_{\infty}\chi_{\tilde{Q}_i}\Big\|_{L^{q^{'}}\big((\lambda_1\omega_2)^{-q^{'}}\big)},\nonumber
	\end{align}
	whenever, for $i=1,\dots, N$, we have $f_{1,i},f_{2,i}\in L^{\infty}(Q_i)$ and $g_i\in L^{\infty}(\tilde{Q}_i)$. Indeed, 
	for certain fixed signs $\sigma_i$ and random sign $\varepsilon_i, \varepsilon^{'}_i$ on probability space with expectation denoted by $\mathbb{E}_{\varepsilon}$ and $\mathbb{E}_{\varepsilon'}$, we have
	\begin{align}
		\sum_{i=1}^{N}\big|\big\langle C_{b}^{k_1}(T)&(f_{1,i},f_{2,i}),g_i\big\rangle\big|
		=
		\sum_{i=1}^{N}\sigma_i\big\langle C_{b}^{k_1}(T)f_{1,i},f_{2,i},g_i\big\rangle\nonumber\nonumber\\
		=&\mathbb{E}_{\varepsilon}\mathbb{E}_{\varepsilon'}\Big\langle C_{b}^{k_1}(T)\Big(\sum_{i=1}^{N}\varepsilon_i \varepsilon^{'}_i f_{1,i},\sum_{j=1}^{N}\varepsilon^{'}_j f_{2,j}\Big),\sum_{k=1}^{N}\varepsilon_k\sigma_kg_k\Big\rangle\nonumber\\
		\le &\mathcal {N}\mathbb{E}_{\varepsilon}\mathbb{E}_{\varepsilon'}\Big\|\sum_{i=1}^{N}\varepsilon_i \varepsilon^{'}_if_{1,i}\Big\|_{L^{p_1}(\omega_1^{p_1})}\nonumber  \Big\|\sum_{i=1}^{N}\varepsilon^{'}_if_{2,i}\Big\|_{L^{p_2}(\omega_2^{p_2})}\Big\|\sum_{i=1}^{N}\varepsilon_i\sigma_ig_i\Big\|_{L^{q^\prime}((\lambda_1\omega_2)^{-q^{'}})}\nonumber\\
		\leq&\mathcal{N}\Big\|\sum_{i=1}^{N}\big|f_{1,i}\big|\Big\|_{L^{p_1}(\omega_1^{p_1})}
		\Big\|\sum_{i=1}^{N}\big|f_{2,i}\big|\Big\|_{L^{p_2}(\omega_2^{p_2})}\Big\|\sum_{i=1}^{N}\big|g_i\big|\Big\|_{L^{q^{'}}((\lambda_1\omega_2)^{-q^{'}})},\nonumber
	\end{align}
	where $\mathcal{N}=\big\|C_{b}^{k_1}(T)\big\|_{L^{p_1}(\omega_1^{p_1})\times L^{p_2}(\omega_2^{p_2})\to L^{q}((\lambda_1\omega_2)^q)}$.
			Assume $v=(\lambda_1\omega_2)^q,  \sigma_1=\omega_1^{-p_1^\prime} ,\sigma_2=\omega_2^{-p_2^\prime}$. 
			For arbitrary $\beta \in \mathbb{R}$ and $x\in \tilde{Q}_0\cap \{b\geq\beta\}$, we have
		\begin{align}
				\int_{Q_0}&(\beta-b)_{+}^{k_1}\sigma_1dy_{1}\nonumber\\
				&\leq\frac{\abs{Q_0}^{2}}{\sigma_2(Q_0)v(\tilde{Q_0}\cap \{b\geq\beta\})}\times\frac{v(\tilde{Q}_0\cap \{b\geq\beta\})\sigma_2(Q_0)}{\abs{Q_0}^2}\int_{Q_0\cap \{b\leqslant \beta\}}|\beta-b|^{k_1}\sigma_1\nonumber\\
			&\lesssim \frac{\abs{Q_0}^{2}}{\sigma_2(Q_0)v(\tilde{Q}_0\cap \{b\geq\beta\})}\big|\big\langle  v\chi_{\tilde{Q}_0\cap\{b\geq \beta\}}, \mathcal{C}_b^{k_1}(T)(\sigma_1\chi_{Q_0\cap\{b\leq\beta\}},\sigma_2\chi_{Q_0})\big\rangle\big|
			\nonumber,
		\end{align}
		In a completely analogous way, we also show that
		\begin{align}
		\int_{Q_0}(b-\beta)_{+}^{k_1}\sigma_1dy_{1}\lesssim\frac{\abs{Q_0}^{2}}{\sigma_2(Q_0)v(\tilde{Q}_0\cap \{b\leq\beta\})}\big|\big\langle  v\chi_{\tilde{Q}_0\cap\{b\leq \beta\}}, \mathcal{C}_b^{k_1}(T)(\sigma_1\chi_{Q_0\cap\{b\geq\beta\}},\sigma_2\chi_{Q_0})\big\rangle\big|\nonumber.
		\end{align}
			Denote set $E_1:=Q_0\cap\{b\leq\beta\},  E_2:=Q_0\cap\{b\geq\beta\}, \tilde{E}_1:= \tilde{Q}_0\cap\{b\geq\beta\}, \tilde{E}_2:= \tilde{Q}_0\cap\{b\leq\beta\}$.
		Let $\beta$ be a median of $b$ in $\tilde{Q}$, we get
		\begin{equation}
			\int_{Q_0}|\beta-b|^{k_1}\sigma_1\lesssim\sum_{i=1}^{2}\frac{\abs{Q_0}^{2}}{\sigma_2(Q_0)v(\tilde{E}_i)}\big|\big \langle v\chi_{\tilde{E}_i}, \mathcal{C}_b^{k_1}(T)(\sigma_1\chi_{{E}_i},\sigma_2\chi_{Q_0})\big\rangle\big|.\nonumber
		\end{equation}
		By triangle inequality and Lemm \ref{key lem3} we see that
		\begin{align}
			\int_{Q}|b-\langle b \rangle_Q|&\lesssim \int_{Q}|b-\langle b \rangle_Q^{\sigma_1^{1/{k_1}}}|\nonumber\\
			&\lesssim\sum_{Q_0\in\mathcal{S}(Q)}\frac{|Q_0|}{\sigma_1^{\frac{1}{k_1}}(Q_0)}\int_{Q_0}|b-\langle b \rangle_{Q_0}^{\sigma_1^{1/{k_1}}}|\sigma_1^{\frac{1}{k_1}}\nonumber\\
			&\lesssim\sum_{Q_0\in\mathcal{S}(Q)}\frac{|Q_0|}{\sigma_1^{\frac{1}{k_1}}(Q_0)}\int_{Q_0}|b-\beta|\sigma_1^{\frac{1}{k_1}}\nonumber\\
			&\lesssim\sum_{Q_0\in\mathcal{S}(Q)}\frac{|Q_0|^{2-\frac{1}{k_1}}}{\sigma_1^{\frac{1}{k_1}}(Q_0)}\Big(\int_{Q_0}|b-\beta|^{k_1}\sigma_1\Big)^{\frac{1}{k_1}}\nonumber.
		\end{align}
		By H\"older inequality we then get
		\begin{align}
			\sum_{Q_0\in\mathcal{S}(Q)}&\frac{|Q_0|^{2-\frac{1}{k_1}}}{\sigma_1^{\frac{1}{k_1}}(Q_0)}\Big(\int_{Q_0}|b-\beta|^{k_1}\sigma_1\Big)^{\frac{1}{k_1}}\nonumber\\
			&\leq\Big(\sum_{Q_0\in\mathcal{S}(Q)}\nu_1(Q_0)\Big)^{\frac{1}{t_1^\prime}}\Bigg(\sum_{Q_0\in\mathcal{S}(Q)}\bigg(\frac{|Q_0|^{2-\frac{1}{k_1}}}{\sigma_1^{\frac{1}{k_1}}(Q_0)\nu_1(Q_0)^{\frac{1}{t_1^\prime}}}\Big(\int_{Q_0}|b-\beta|^{k_1}\sigma_1\Big)^{\frac{1}{k_1}}\bigg)^{t_1}\Bigg)^{\frac{1}{t_1}}\nonumber\\
			&\lesssim \nu_1(Q)^{\frac{1}{t_1^\prime}}\Bigg(\sum_{Q_0\in\mathcal{S}(Q)}\bigg(\frac{|Q_0|^{2-\frac{1}{k_1}}}{\sigma_1^{\frac{1}{k_1}}(Q_0)\nu_1(Q_0)^{\frac{1}{t_1^\prime}}}\Big(\int_{Q_0}|b-\beta|^{k_1}\sigma_1\Big)^{\frac{1}{k_1}}\bigg)^{t_1}\Bigg)^{\frac{1}{t_1}}\nonumber\\
			&=:\nu_1(Q)^{\frac{1}{t_1^\prime}}h^{\frac{1}{t_1}}\nonumber.
		\end{align}
		According to the Proposition \ref{prop} and introduce an enumeration of the sparse collection $\mathcal{S}$, we have
		\begin{align}
			\Big\|&	\sum_{Q\in\mathcal{S}} \frac{\chi_{Q}(x)}{\nu_1(Q)}\int_{Q} \big|b- \langle b \rangle_{Q}\big|\Big\|_{L^{t_1}(\nu_1)}\nonumber\\
			&=\sup\limits_{\|g\|_{L^{t_1^{'}}(\nu_1)}=1} \sum_{Q\in\mathcal{S}}\frac{1}{\nu_1(Q)}\Big(\int_{Q} \big|b-\langle b \rangle _{Q}\big|\Big)  \nu_1(Q) \big\langle \abs{g} \big\rangle_{Q}^{\nu_1}\nonumber\\
			&\lesssim \sup\limits_{\|g\|_{L^{t_1^{'}}(\nu_1)}=1}	\sum_{Q\in\mathcal{S}} h^{\frac{1}{t_1}}\nu_1(Q)^{\frac{1}{t_1^\prime}}\big\langle \abs{g} \big\rangle_{Q}^{\nu_1} \nonumber\\
			&\leq\Big(\sum_{Q\in\mathcal{S}} \big(\langle \abs{g} \rangle_{Q_j}^{\nu_1}\big)^{t_1^{'}}\nu_1(Q_j)\Big)^{\frac{1}{t_1^{'}}}\Big(\sum_{Q\in\mathcal{S}}h\Big)^{\frac{1}{t_1}}\nonumber.
		\end{align}	
		 We can estimate the first factor by
		\begin{align}
			\sum_{Q\in\mathcal{S}} \Big(\big\langle g \big\rangle_{Q}^{\nu_1}\Big)^{t_1^{'}}\nu_1(Q)\lesssim	\sum_{Q\in\mathcal{S}} \inf\limits_{z\in Q}M_{\nu_1}g(z)^{t_1^{'}}\nu_1(E(Q))\lesssim \int \big(M_{\nu_1}g(z)\big)^{t_1^{'}}\nu_1 \lesssim \int g^{t_1^{'}}\nu_1.\nonumber
		\end{align}
		For the second factor, choose $r={t_1}/{k_1}={-1}/{(\alpha_1k_1)}$
		\begin{align}
			\Big(\sum_{Q\in\mathcal{S}}h\Big)^{\frac{1}{t_1}}
			&\lesssim\Big(\sum_{Q\in\mathcal{S}}\sum_{Q_0\in\mathcal{S}(Q)}\sum_{i=1}^{2}\Big( \frac{|Q_0|^{2k_1+1}\big|\big \langle v\chi_{\tilde{E}_i}, \mathcal{C}_b^{k_1}(T)(\sigma_1\chi_{{E}_i},\sigma_2\chi_{Q_0})\big\rangle\big|}{\sigma_1^{\frac{1}{k_1}}(Q_0)^{k_1}\nu_1(Q_0)^{k_1+k_1\alpha_1}\sigma_2(Q_0)v(\tilde{E}_i)}\Big)^{\frac{t_1}{k_1}}\Big)^{\frac{1}{t_1}}\nonumber\\
			&\lesssim \Big(\sum_{j=1}^{\infty}\Big( \frac{|Q_j|^{2k_1+1}\big|\big \langle v\chi_{\tilde{E}^j_i}, \mathcal{C}_b^{k_1}(T)(\sigma_1\chi_{{E}^j_i},\sigma_2\chi_{Q_j})\big\rangle\big|}{\sigma_1^{\frac{1}{k_1}}(Q_j)^{k_1}\nu_1(Q_j)^{k_1+k_1\alpha_1}\sigma_2(Q_j)v(\tilde{E}^j_i)}\Big)^{r}\Big)^{\frac{1}{rk_1}}\nonumber.
		\end{align}
		It is enough to give a uniform bound for the finite sum. We dualise this term with $\sum_{j=1}^{N} \lambda_{Q_j}^{r^{'}}\leq1$, then we have
		\begin{align}
			\Big(\sum_{j=1}^{N}&\big(\frac{|Q_j|^{2k_1+1} |\langle v\chi_{\tilde{E}^j_i}, \mathcal{C}_b^{k_1}(T)(\sigma_1\chi_{{E}^j_i},\sigma_2\chi_{Q_j})\big\rangle\big|}{\nu_1(Q_j)^{k_1+\alpha_1k_1}\sigma_1^{\frac{1}{k_1}}(Q_j)^{k_1}\sigma_2(Q_j)v(\tilde{E}^j_i)}\big)^{r} \Big)^{\frac{1}{rk_1}}\nonumber\\
			=&\Big(\sum_{j=1}^{N}\frac{|Q_j|^{2k_1+1} |\langle v\chi_{\tilde{E}^j_i}, \mathcal{C}_b^{k_1}(T)(\sigma_1\chi_{{E}^j_i},\sigma_2\chi_{Q_j})\big\rangle\big|}{\nu_1(Q_j)^{k_1+\alpha_1k_1}\sigma_1^{\frac{1}{k_1}}(Q_j)^{k_1}\sigma_2(Q_j)v(\tilde{E}^j_i)}\lambda_{Q_j} \Big)^{\frac{1}{k_1}}\nonumber\\
			=& \Big(\sum_{j=1}^{N}\frac{|Q_j|^{2k_1+1} \big|\big\langle v\chi_{\tilde{E}^j_i}\lambda_{Q_j}^{\frac{r^\prime}{q^\prime}}, \mathcal{C}_b^{k_1}(T)(\sigma_1\chi_{{E}^j_i}\lambda_{Q_j}^{\frac{r^\prime}{p_1}},\sigma_2\chi_{Q_j}\lambda_{Q_j}^{\frac{r^\prime}{p_2}})\big\rangle\big|}{\nu_1(Q_j)^{k_1\alpha_1+k_1}\sigma_1^{\frac{1}{k_1}}(Q_j)^{k_1}\sigma_2(Q_j)v(\tilde{E}^j_i)}\Big)^{\frac{1}{k_1}}\nonumber \\
			=&\Big(\sum_{j=1}^{N}\big|\big\langle v\chi_{\tilde{E}^j_i}\lambda_{Q_j}^{\frac{r^\prime}{q^\prime}}v(Q_j)^{-\frac{1}{q^\prime}}, \mathcal{C}_b^{k_1}(T)(\sigma_1\chi_{{E}^j_i}\lambda_{Q_j}^{\frac{r^\prime}{p_1}}\sigma_1(Q_j)^{-\frac{1}{p_1}},\sigma_2\chi_{Q_j}\lambda_{Q_j}^{\frac{r^\prime}{p_2}}\sigma_2(Q_j)^{-\frac{1}{p_2}})\big\rangle\big|\nonumber\\
			&\quad\quad\times\frac{|Q_j|^{2k_1+1} v(Q_j)^{\frac{1}{q^\prime}}\sigma_1(Q_j)^{\frac{1}{p_1}}\sigma_2(Q_j)^{\frac{1}{p_2}}}{\nu_1(Q_j)^{k_1\alpha_1+k_1}\sigma_1^{\frac{1}{k_1}}(Q_j)^{k_1}\sigma_2(Q_j)v(\tilde{E}^j_i)}\Big)^{\frac{1}{k_1}},\nonumber
		\end{align}
		since $v, \sigma_1$ are $A_\infty$ weight and $\beta$ is median of $b$, we have
		$$v(\tilde{E}^j_i)\thicksim v(\tilde{Q}_j)\thicksim_A v(Q_j), \sigma_1(E^j_i)\thicksim\sigma_1(Q_j).$$
		Hence
		\begin{align*}
			\frac{|Q_j|^{2k_1+1} v(Q_j)^{\frac{1}{q^\prime}}\sigma_1(Q_j)^{\frac{1}{p_1}}\sigma_2(Q_j)^{\frac{1}{p_2}}}{\nu_1(Q_j)^{k_1\alpha_1+k_1}\sigma_1^{\frac{1}{k_1}}(Q_j)^{k_1}\sigma_2(Q_j)v(\tilde{E}^j_i)} 
			\simeq\frac{ \langle v \rangle_{Q_j}^{\frac{1}{q^\prime}} \langle \sigma_1 \rangle_{Q_j}^{\frac{1}{p_1}} \langle \sigma_2 \rangle_{Q_j}^{\frac{1}{p_2}}}{ \langle \nu_1 \rangle_{Q_j}^{k_1\alpha_1+k_1} \langle \sigma_1^{\frac{1}{k_1}} \rangle_{Q_j}^{k_1} \langle \sigma_2 \rangle_{Q_j} \langle v \rangle_{Q_j}}\nonumber
		\end{align*}
		Since $1+\alpha_1<1$, we apply H\"older's inequality to estimate $\langle \nu_1^{1+\alpha_1} \rangle_{Q_j}$. Furthermore, since $\sigma_1\in A_{2p_1^\prime}$, we   use Lemma \ref{key lem2} for $\langle \sigma_1\rangle_{Q_j}$. Then,
		\begin{equation}
			\frac{ \langle v \rangle_{Q_j}^{\frac{1}{q^\prime}} \langle \sigma_1 \rangle_{Q_j}^{\frac{1}{p_1}} \langle \sigma_2 \rangle_{Q_j}^{\frac{1}{p_2}}}{ \langle \nu_1 \rangle_{Q_j}^{k_1\alpha_1+k_1} \langle \sigma_1^{\frac{1}{k_1}} \rangle_{Q_j}^{k_1} \langle \sigma_2 \rangle_{Q_j} \langle v \rangle_{Q_j}}
			\lesssim\frac{ \langle v \rangle_{Q_j}^{\frac{1}{q^\prime}} \langle \sigma_1 \rangle_{Q_j}^{\frac{1}{p_1}} \langle \sigma_2 \rangle_{Q_j}^{\frac{1}{p_2}}}{ \langle \nu_1^{1+\alpha_1} \rangle_{Q_j}^{k_1} \langle \sigma_1 \rangle_{Q_j} \langle \sigma_2 \rangle_{Q_j} \langle v \rangle_{Q_j}}.\nonumber
		\end{equation}
		 Lemma \ref{key lem2} can be applied under the condition that  $1\in A_{\infty}$,
		 \begin{equation}
		 \frac{ \langle v \rangle_{Q_j}^{\frac{1}{q^\prime}} \langle \sigma_1 \rangle_{Q_j}^{\frac{1}{p_1}} \langle \sigma_2 \rangle_{Q_j}^{\frac{1}{p_2}}}{ \langle \nu_1^{1+\alpha_1} \rangle_{Q_j}^{k_1} \langle \sigma_1 \rangle_{Q_j} \langle \sigma_2 \rangle_{Q_j} \langle v \rangle_{Q_j}}
		 = \frac{ 1}{ \langle \nu_1^{1+\alpha_1} \rangle_{Q_j}^{k_1} \langle \sigma_1 \rangle_{Q_j}^{\frac{1}{p_1^\prime}} \langle \sigma_2 \rangle_{Q_j}^{\frac{1}{p_2^\prime}} \langle v \rangle_{Q_j}^{\frac{1}{q}}}
		 \lesssim 1. \nonumber
		 \end{equation}
		  Now, we arrive at
		\begin{align}
			\sum_{j=1}^{N}&\frac{|Q_j|^{2k_1+1} \big|\big\langle v\chi_{\tilde{E}^j_i}\lambda_{Q_j}^{\frac{r^\prime}{q^\prime}}, \mathcal{C}_b^{k_1}(T)(\sigma_1\chi_{{E}^j_i}\lambda_{Q_j}^{\frac{r^\prime}{p_1}},\sigma_2\chi_{Q_j}\lambda_{Q_j}^{\frac{r^\prime}{p_2}})\big\rangle\big|}{\nu_1(Q_j)^{k_1\alpha_1+k_1}\sigma_1^{\frac{1}{k_1}}(Q_j)^{k_1}\sigma_2(Q_j)v(\tilde{E}^j_i)}\nonumber\\
			&\lesssim\sum_{j=1}^{N}\big|\big\langle v\chi_{\tilde{E}^j_i}\lambda_{Q_j}^{\frac{r^\prime}{q^\prime}}v(Q_j)^{-\frac{1}{q^\prime}}, \mathcal{C}_b^{k_1}(T)(\sigma_1\chi_{{E}^j_i}\lambda_{Q_j}^{\frac{r^\prime}{p_1}}\sigma_1(Q_j)^{-\frac{1}{p_1}},\sigma_2\chi_{Q_j}\lambda_{Q_j}^{\frac{r^\prime}{p_2}}\sigma_2(Q_j)^{-\frac{1}{p_2}})\big\rangle\big|\nonumber\\
			&\leq\mathcal{N} \Big\|\sum_{j=1}^{N}v\chi_{\tilde{E}^j_i}\lambda_{Q_j}^{\frac{r^\prime}{q^\prime}}v(Q_j)^{-\frac{1}{q^\prime}}\Big\|_{L^{q^{'}}\big((\lambda_1\omega_2)^{-q^{\prime}}\big)}\times\Big\|\sum_{j=1}^{N}\sigma_1\chi_{{E}^j_i}\lambda_{Q_j}^{\frac{r^\prime}{p_1}}\sigma_1(Q_j)^{-\frac{1}{p_1}}\Big\|_{L^{p_1}(\omega_1^{p_1})}\nonumber\\
			&\quad\times\Big\|\sum_{j=1}^{N}\sigma_2\chi_{Q_j}\lambda_{Q_j}^{\frac{r^\prime}{p_2}}\sigma_2(Q_j)^{-\frac{1}{p_2}}\Big\|_{L^{p_2}(\omega_2^{p_2})}.\nonumber
		\end{align}
		Hence, we use the disjoint major subsets $E(Q_j)\subset Q_j$,
		\begin{equation}
			\Big\|\sum_{j=1}^{N}\sigma_2\chi_{Q_j}\lambda_{Q_j}^{\frac{r^\prime}{p_2}}\sigma_2(Q_j)^{-\frac{1}{p_2}}\Big\|_{L^{p_2}(\omega_2^{p_2})}\lesssim \Big(\sum_{j=1}^{N}\lambda_{Q_j}^{r^{'}}\sigma_2(Q_j)^{-1}\sigma_2\big(E(Q_j)\big)\Big)^{\frac{1}{p_2}}\lesssim 1, \nonumber
		\end{equation}   
		the other two factors are similar and  then finishing the prove.
\end{proof}
	Next, we will prove $\mathcal{C}_{b}^{k_1}(T)$ is $L^{p_1}(\omega_1^{p_1})\times L^{p_2}(\omega_2^{p_2})\rightarrow L^q((\lambda_1\omega_2)^{q})$ bounded with bloom weight. First in Theorem \ref{mmain thm}, we consider one of the sparse forms in the conclusion of Theorem \ref{main thm}.
	Following lemmas will be needed to prove Theorem \ref{mmain thm}.
\begin{lem}\emph{(}\cite[(2.4)]{MR2086703}\emph{)}\label{lem 11}
	Let $p\in [1,\infty)$ and $\lambda_Q\geq0$ for all $Q\in \mathscr{D}$. Then we have
	$$\Big(\sum_{Q\in\mathscr{D}}\lambda_Q\chi_Q\Big)^p\leq p\sum_{Q\in\mathscr{D}}\lambda_Q\chi_Q\Big(\sum_{Q^\prime\in\mathscr{D}:Q^\prime\subset Q}\lambda_{Q^\prime}\chi_{Q^\prime}\Big)^{p-1}.$$
\end{lem}
\begin{lem}\emph{(}\cite[Proposition 2.2]{MR2086703}\emph{)}\label{lem 12}
	Let $p\in(1,\infty]$, let $\omega$ be a weight and let $\lambda_Q\geqslant0$ for all $Q\in\mathscr{D}$. Then
	$$\Big\|\sum_{Q\in\mathscr{D}}\lambda_Q\chi_Q\Big\|_{L^p(\omega)}\simeq\Big(\sum_{Q\in\mathscr{D}}\lambda_Q\big(\frac{1}{\omega(Q)}\sum_{Q^\prime\in\mathscr{D},Q^\prime\subset Q}\lambda_{Q^\prime}\omega(Q^\prime)\big)^{p-1}\omega(Q)\Big)^{\frac{1}{p}}.$$
\end{lem}
\begin{lem}\emph{(}\cite[Theorem 2.2]{MR3000426}\emph{)}\label{reference lemma1}
	If $0\leqslant \alpha <1, 1<p\leqslant {1}/{\alpha}$, and ${1}/{q}={1}/{p}-\alpha$, then
	\begin{align}
		\|M^{\mathscr{D}}_{\alpha,\mu}(f)\|_{L^q(\mu)}=\Big\|\sup \limits_{Q\in\mathscr{D}} \frac{\chi_{Q}}{\mu(Q)^{1-\alpha}}\int_{Q}|f|d\mu \Big\|_{L^q(\mu)}\lesssim\|f\|_{L^p(\mu)}.\nonumber
	\end{align}
\end{lem}
\begin{lem}\label{lem 13}
	Let $1<p\leqslant q <\infty$, set $\alpha={1}/{p}-{1}/{q}$ and let $\zeta\in A_{\infty}$. For any sparse family of $\mathcal{S} \subset \mathscr{D}$ and $f\in L^p(\zeta)$ we have
	$$\big\|\sum_{Q\in \mathcal{S}} \langle |f| \rangle ^{\zeta}_{\frac{1}{1+\alpha},Q}\zeta(Q)^{\alpha}\chi_{Q}\big\|_{L^q(\zeta)}\lesssim\|f\|_{L^p(\zeta)} .$$
\end{lem}
\begin{proof}
	\begin{align}
		\big\|\sum_{Q\in \mathcal{S}} \langle |f| \rangle ^{\zeta}_{\frac{1}{1+\alpha},Q}\zeta(Q)^{\alpha}\chi_{Q}\big\|_{L^q(\zeta)}&=\sup\limits_{\|g\|_{L^{q^\prime}(\zeta)}\leqslant1}\sum_{Q\in \mathcal{S}} \langle |f| \rangle ^{\zeta}_{\frac{1}{1+\alpha},Q}\zeta(Q)^{\alpha}\langle g \rangle ^{\zeta}_Q \zeta(Q)\nonumber\\
		&\lesssim \sup\limits_{\|g\|_{L^{q^\prime}(\zeta)}\leqslant1}\sum_{Q\in \mathcal{S}} \int_{E(Q)}\langle M_{\frac{1}{1+\alpha},\zeta}(f) \rangle ^{\zeta}_{Q}\zeta(Q)^{\alpha}\langle g \rangle ^{\zeta}_Q \zeta\nonumber\\
		&\leq \|M_{\frac{1}{1+\alpha},\zeta}(f)\|_{L^p(\zeta)}\cdot\|M^{\mathscr{D}}_{\alpha,\zeta}(g)\|_{L^{p^\prime}(\zeta)}\nonumber\\
		&\lesssim \|f\|_{L^p(\zeta)}\nonumber.
	\end{align}
\end{proof}
	\begin{thm}\label{mmain thm}Let $1\leqslant r_1<p_1,q_1\leqslant \infty, 1\leqslant r_2 <p_2<\infty, 1<p,q <s \leqslant \infty, {1}/{p_1}+{1}/{p_2}={1}/{p} , {1}/{q_1}+{1}/{p_2}={1}/{q}, k_1\in \mathbb{N} $ and $b\in L_{\rm{loc}}^1(\mathbb{R}^n) $. 
	Let $\vec{q}=(q_1,p_2) , \vec{p}=(p_1, p_2 ), \vec{r}=(r_1,r_2,s^\prime)$. Assume that $(\lambda_1, \omega_2 )\in A_{\vec{q},\vec{r}}, (\omega_1, \omega_2 )\in A_{\vec{p},\vec{r}} $. For any sparse family $\mathcal{S} \subset \mathscr{D}$ , $f_1\in L^{p_1}(\omega_1^{p_1})$, $f_2  \in L^{p_2}(\omega_2^{p_2}) $ and $g\in L^{q^{\prime}}\big((\lambda_1 \omega_2)^{-q^{\prime}}\big)$, for Bloom weight $\nu_1$ defined by (\ref{ fomulater1}), $\nu_1\in A_{\infty}$ and $\alpha_1:=  -{1}/{t_1}:={1}/{p_1k_1}-{1}/{q_1k_1}$,  we have
		\begin{align}
			\sum_{Q\in \mathcal{S}}&\big\langle {|b-\langle{b}\rangle_{Q}|^{k_1}|f_1|}\big\rangle_{r_1,Q}\langle |f_2|\rangle_{r_2,Q}\langle |g|\rangle_{s^\prime,Q}|Q|\leqslant  C(\omega_1,\omega_2,\lambda_1) \nonumber\\
			&\quad\quad\quad\times \|f_1\omega_1\|_{L^{p_1}}\|f_2\omega_2\|_{L^{p_2}}\big\|g(\lambda_1 \omega_2)^{-1}\big\|_{L^{q^{\prime}}}
			\begin{cases}
				\|b\|_{BMO_{\nu_1}^{\alpha_1}}^{k_1},\quad& p_1\leqslant q_1,\nonumber\\
				\big\|M_{\nu_1} ^\#(b)\big\|_{L^{t_1}(\nu_1)}^{k_1},\quad &q_1< p_1,
			\end{cases}
		\end{align}
	\end{thm}
	\begin{proof}[Proof of Theorem \ref{mmain thm}]
			\textbf{The case  $p_1\leqslant q_1$.} By Lemma \ref{key lem3}, there exists a sparse collection $\mathcal{S}\subset \mathcal{S}^{\prime}\subset  \mathscr{D}$ such that for any $Q\in \mathcal{S}$,
		\begin{align}
			\big\langle {|b-\langle{b}\rangle_{Q}|^{k_1}|f_1|}\big\rangle_{r_1,Q}^{r_1} &\lesssim \frac{1}{|Q|}\int_{Q} \Big(\sum_{P\in \mathcal{S}^\prime:P\subseteq Q}\frac{1}{|P|}\int_{P}|b-\langle b \rangle_{P}|\chi_{P}\Big)^{r_1k_1}|f_1|^{r_1}\nonumber\\
			&\leq \|b\|^{r_1k_1}_{BMO_{\nu_1}^{\alpha_1 }}\frac{1}{|Q|}\int_{Q} \Big(\sum_{P\in \mathcal{S}^\prime:P\subseteq Q}\frac{\nu_1(P)^{1+\alpha_1}}{|P|}\chi_P\Big)^{r_1k_1}|f_1|^{r_1}\nonumber.
		\end{align}
		Let $k:=\lfloor k_1r_1 \rfloor $ and $\gamma:=r_1k_1-(k-1)$. Applying Lemma \ref{lem 11} $(k-1)$ times, we have
		\begin{align}
			\int_{Q} \Big(\sum_{P\in \mathcal{S}^\prime:P\subseteq Q}\frac{\nu_1(P)^{1+\alpha_1}}{|P|}\chi_P\Big)^{r_1k_1}|f_1|^{r_1}
			\lesssim&\sum_{\substack{P_{k-1}\subseteq \cdots \subseteq P_1\subseteq Q\\P_1,\dots,P_{k}\in\mathcal{S}^\prime}}\frac{\nu_1(P_1)^{1+\alpha_1}}{|P_1|}\cdots\frac{\nu_1(P_{k-1})^{1+\alpha_1}}{|P_{k-1}|}\nonumber\\
			&\quad\times\int_{P_{k-1}}\Big(\sum_{P_{k}\subseteq P_{k-1}} \frac{\nu_1(P_{k})^{1+\alpha_1}}{|P_{k}|}\chi_{P_{k}}\Big)^{\gamma}|f_1|^{r_1}\nonumber.
		\end{align}
		Then, we denote
			$$\mathcal{A}^{\zeta}_{\alpha,\mathcal{S}^\prime}(\varphi):=\sum_{Q \in \mathcal{S}^\prime}\langle |\varphi| \rangle^{\zeta}_{\frac{1}{1+\alpha},Q}\zeta(Q)^{\alpha}\chi_{Q}.$$
		 We omit $\alpha$, when $\alpha=0$. Let $\zeta_1,\eta_1,\eta_2\in A_{\infty}$, which will be determined later. Using Lemma  \ref{lem 12} and Lemma \ref{key lem1}, we obtain
		\begin{align}
			\int_{P_{k-1}}&\Big(\sum_{P_{k}\subseteq P_{k-1}} \frac{\nu_1(P_{k})^{1+\alpha_1}}{|P_{k}|}\chi_{P_{k}}\Big)^{\gamma}|f_1|^{r_1}\nonumber\\
			&\simeq\sum_{P_{k}\subseteq P_{k-1}}\frac{\nu_1(P_{k})^{1+\alpha_1}}{|P_{k}|}\Big(\int_{P_k}|f_1|^{r_1}\Big)^{2-\gamma}\Big(\sum_{P\subseteq P_{k}}\frac{\nu_1(P)^{1+\alpha_1}}{|P|}\int_{P}|f_1|^{r_1}\Big)^{\gamma-1}\nonumber\\
			&\lesssim \sum_{P_{k}\subseteq P_{k-1}}\frac{\nu_1(P_{k})^{1+\alpha_1}}{|P_{k}|}\Big(\int_{P_k}|f_1|^{r_1}\Big)^{2-\gamma}\Big(\sum_{P\subseteq P_k}\nu_1\zeta_1^{\frac{1}{1+\alpha_1}}(P)^{1+\alpha_1}\langle|f_1|^{r_1}\zeta_1^{-1}\rangle_{P}^{\zeta_1}\Big)^{\gamma-1}\nonumber,
		\end{align}
		using Minkowski's inequality,
		\begin{align}
		\sum_{P\subseteq P_k}\nu_1\zeta_1^{\frac{1}{1+\alpha_1}}(P)&^{1+\alpha_1}\langle|f_1|^{r_1}\zeta_1^{-1}\rangle_{P}^{\zeta_1}=	\sum_{P\subseteq P_k}\Big(\int_{P}\big(\langle|f_1|^{r_1}\zeta_1^{-1}\rangle_{P}^{\zeta_1}\Big)^{\frac{1}{1+\alpha_1}}\nu_1\zeta_1^{\frac{1}{1+\alpha_1}}\Big)^{1+\alpha_1}\nonumber\\
		&\leq \Big(\int_{P_k}\Big(	\sum_{P\subseteq P_k}(\langle|f_1|^{r_1}\zeta_1^{-1}\rangle_{P}^{\zeta_1})\chi_{P}\nu_1^{1+\alpha_1}\zeta_1\Big)^{\frac{1}{1+\alpha_1}}\Big)^{1+\alpha_1}\nonumber\\
		&=\Big(\int_{P_k} \big(\mathcal{A}^{\zeta_1}_{\mathcal{S}^{\prime}}(f_1^{r_1}\cdot\zeta_1^{-1})\nu_1^{1+\alpha_1}\zeta_1\big)^{\frac{1}{1+\alpha_1}}\Big)^{(1+\alpha_1)}
		:=\Big(\int_{P_k} h_0^{\frac{1}{1+\alpha_1}}\Big)^{(1+\alpha_1)}\nonumber.
		\end{align}
		By a similar method as above
		\begin{align}
			\sum_{P_{k}\subseteq P_{k-1}}&\frac{\nu_1(P_{k})^{1+\alpha_1}}{|P_{k}|}\Big(\int_{P_k}|f_1|^{r_1}\Big)^{2-\gamma}\Big(\int_{P_k} h_0^{\frac{1}{1+\alpha_1}}\Big)^{(1+\alpha_1)(\gamma-1)}\nonumber\\
			&=\sum_{P_{k}\subseteq P_{k-1}}\frac{\nu_1(P_k)^{1+\alpha_1}\eta_1(P_k)^{2-\gamma}\eta_2(P_k)^{\gamma-1}}{|P_k|}\Big(\langle |f_1|^{r_1}\eta_1^{-1}\rangle_{P_k}^{\eta_1}\Big)^{2-\gamma}\nonumber\\
			&\quad\quad\quad\quad\quad\cdot\Big(\langle h_0\eta_2^{-(1+\alpha_1)}\rangle^{\eta_2}_{\frac{1}{1+\alpha_1},P_k}\Big)^{\gamma-1}|\eta_2(P_k)|^{\alpha_1(\gamma-1)}\nonumber\\
			&\lesssim \sum_{P_{k}\subseteq P_{k-1}}\nu_1\eta_1^{\frac{2-\gamma}{1+\alpha_1}}\eta_2^{\frac{\gamma-1}{1+\alpha_1}}(P_k)^{1+\alpha_1}\Big(\langle |f_1|^{r_1}\eta_1^{-1}\rangle_{P_k}^{\eta_1}\Big)^{2-\gamma}\nonumber\\
			&\quad\quad\quad\quad\quad\cdot\Big(\langle h_0\eta_2^{-(1+\alpha_1)}\rangle^{\eta_2}_{\frac{1}{1+\alpha_1},P_k}\Big)^{\gamma-1}|\eta_2(P_k)|^{\alpha_1(\gamma-1)}\nonumber\\
			&\lesssim \bigg(\int_{P_{k-1}}\Big(\sum_{P_k\subseteq P_{k-1}}\chi_{P_k}(\langle |f_1|^{r_1}\eta_1^{-1}\rangle_{P_k}^{\eta_1})^{2-\gamma}(\langle h_0\eta_2^{-(1+\alpha_1)}\rangle^{\eta_2}_{\frac{1}{1+\alpha_1},P_k})^{\gamma-1}\nonumber\\
			&\quad\quad\quad\quad\quad\cdot|\eta_2(P_k)|^{\alpha_1(\gamma-1)}\nu_1^{1+\alpha_1}\eta_1^{2-\gamma}\eta_2^{\gamma-1}\Big)^{\frac{1}{1+\alpha_1}}\bigg)^{1+\alpha_1}\nonumber.
		\end{align}
			 According to H\"older's inequality,
			 \begin{align}
			 	\Big(\int_{P_{k-1}}&\big(\sum_{P_k\subseteq P_{k-1}}\chi_{P_k}(\langle |f_1|^{r_1}\eta_1^{-1}\rangle_{P_k}^{\eta_1})^{2-\gamma}(\langle h_0\eta_2^{-(1+\alpha_1)}\rangle^{\eta_2}_{\frac{1}{1+\alpha_1},P_k})^{\gamma-1}\nonumber\\
			 	&\quad\quad\cdot|\eta_2(P_k)|^{\alpha_1(\gamma-1)}\nu_1^{1+\alpha_1}\eta_1^{2-\gamma}\eta_2^{\gamma-1}\big)^{\frac{1}{1+\alpha_1}}\Big)^{1+\alpha_1}\nonumber\\
			 	\leq&\Big(\int_{P_{k-1}}\big(\mathcal{A}^{\eta_1}_{\mathcal{S}^\prime}(f_1^{r_1}\cdot\eta_1^{-1})^{2-\gamma}\mathcal{A}^{\eta_2}_{\alpha_1,\mathcal{S}^\prime}(h_0\cdot\eta_2^{-(1+\alpha_1)})^{\gamma-1}\nu_1^{1+\alpha_1}\eta_1^{2-\gamma}\eta_2^{\gamma-1}\big)^{\frac{1}{1+\alpha_1}}\Big)^{1+\alpha_1}\nonumber\\
			 	:=&\Big(\int_{P_{k-1}}h_1^{\frac{1}{1+\alpha_1}}\Big)^{1+\alpha_1}\nonumber,
			 \end{align}
			 Let $\zeta_2,\dots,\zeta_{k}\in A_{\infty}$ which will be determined later. Using Lemma \ref{key lem1} and Minkowski's inequality, we have
		\begin{align}
			\sum_{P_{k-1}\subseteq P_{k-2}}&\frac{\nu_1(P_{k-1})^{1+\alpha_1}}{|P_{k-1}|} \Big(\int_{P_{k-1}}h_1^{\frac{1}{1+\alpha_1}}\Big)^{1+\alpha_1}\nonumber\\
			=&\sum_{P_{k-1}\subseteq P_{k-2}}\frac{\nu_1(P_{k-1})^{1+\alpha_1}\zeta_2(P_{k-1})^{1+\alpha_1}}{|P_{k-1}|} \Big(\frac{1}{\zeta_2(P_{k-1})}\int_{P_{k-1}}h_1^{\frac{1}{1+\alpha_1}}\zeta_2^{-1}\zeta_2\Big)^{1+\alpha_1}\nonumber\\
			\lesssim&\sum_{P_{k-1}\subseteq P_{k-2}} \nu_1\zeta_2^{\frac{1}{1+\alpha_1}}(P_{k-1})^{1+\alpha_1}\langle h_1 \zeta_2^{-(1+\alpha_1)}\rangle^{\zeta_2}_{\frac{1}{1+\alpha_1},P_{k-1}}\zeta_2(P_{k-1})^{\alpha_1}\nonumber\\
			\leq &\Big( \int_{P_{k-2}}\big(\mathcal{A}^{\zeta_2}_{\alpha_1,\mathcal{S}^\prime}(h_1\zeta_2^{-(1+\alpha_1)})\nu_1^{1+\alpha_1}\zeta_2\big)^{\frac{1}{1+\alpha_1}}\Big)^{1+\alpha_1}
			:=\Big(\int_{P_{k-2}}h_2^{\frac{1}{1+\alpha_1}}\Big)^{1+\alpha_1}\nonumber,
		\end{align}
	    denote
		$$h_i:=\mathcal{A}^{\zeta_i}_{\alpha_1,\mathcal{S}^\prime}(h_{i-1}\zeta_i^{-(1+\alpha_1)})\nu_1^{1+\alpha_1}\zeta_i,\quad\quad i=2,\dots,k.$$
	    By iterating this process we arrive at
		$$\int_{Q} \Big(\sum_{P\in \mathcal{S}^\prime:P\subseteq Q}\frac{\nu_1(P)^{1+\alpha_1}}{|P|}\chi_P\Big)^{r_1k_1}|f_1|^{r_1}\lesssim \Big(\int_Q h_{k}^{\frac{1}{1+\alpha_1}}\Big)^{1+\alpha_1}.$$
		Now, we conclude
		\begin{align}
			&\sum_{Q\in \mathcal{S}}\big\langle {|b-\langle{b}\rangle_{Q}|^{k_1}|f_1|}\big\rangle_{r_1,Q}\langle |f_2|\rangle_{r_2,Q}\langle |g|\rangle_{s^\prime,Q}|Q|\nonumber\\
			&\quad\quad\quad\quad\lesssim\|b\|^{k_1}_{BMO_{\nu_1}^{\alpha_1 }}\sum_{Q\in \mathcal{S}} \langle h_{k}^{\frac{1}{r_1}}\rangle_{\frac{r_1}{1+\alpha_1},Q}\langle f_2 \rangle_{r_2,Q} \langle g\rangle_{s^\prime,Q} |Q|^{1+\frac{\alpha_1}{r_1}},\nonumber
		\end{align}
		by Theorem \ref{mmmain thm}, let ${1}/{u_1}={1}/{q_1}+{\alpha_1}/{r_1}$, implies
		\begin{align}
			\sum_{Q\in \mathcal{S}}&\big\langle {|b-\langle{b}\rangle_{Q}|^{k_1}|f_1|}\big\rangle_{Q}\langle |f_2|\rangle_{Q}\langle |g|\rangle_{Q}|Q|\nonumber\\
			&\quad\quad\lesssim\|b\|^{k_1}_{BMO_{\nu_1}^{\alpha_1 }}\|h_{k}^{\frac{1}{r_1}}\lambda_1\|_{L^{u_1}}\|f_2\omega_2\|_{L^{p_2}}\|g(\lambda_1\omega_2)^-1{}\|_{L^{q^\prime}}\nonumber,
		\end{align}
		For term $\|h_{k}\lambda_1^{r_1}\|_{L^{u_1/r_1}}$,  define 
		$$\frac{1}{u_i}:=\frac{1}{q_1}+i\frac{\alpha_1}{r_1}=\frac{i}{r_1k_1p_1}+(1-\frac{i}{r_1k_1})\frac{1}{q_1},\quad i=1,\dots,k,$$
		and we need
		\begin{enumerate}
			\item  $(\nu_1^{1+\alpha_1}\zeta_k)^{\frac{u_1}{r_1}}\lambda_1^{u_1}=\zeta_k$;
			\item 
			$(\nu_1^{1+\alpha_1}\zeta_{i-1}\zeta_i^{-(1+\alpha_1)})^{\frac{u_{k+2-i}}{r_1}}\zeta_i=\zeta_{i-1},\quad i=3,\dots,k$.
		\end{enumerate}
		These give us
		\begin{align}
		\zeta_{k}&=\Big(\nu_1^{(1+\alpha_1)}\lambda_1^{r_1}\Big)^{\frac{u_1}{r_1-u_1}}=\Big(\lambda_1^{\frac{r_1k_1-1}{k_1}}\omega_1^{\frac{1}{k_1}}\Big)^{\frac{u_1}{r_1-u_1}}\nonumber\\
		&=\Big(\lambda_1^{\frac{r_1q_1}{r_1-q_1}\frac{(r_1-q_1)(r_1k_1-1)}{k_1r_1q_1}}\omega_1^{\frac{r_1p_1}{r_1-p_1}\frac{r_1-p_1}{k_1r_1p_1}}\Big)^{\frac{u_1}{r_1-u_1}}\nonumber,\nonumber
		\end{align}
		since $$\frac{(r_1-q_1)(r_1k_1-1)}{k_1r_1q_1}+\frac{r_1-p_1}{k_1r_1p_1}=\frac{r_1}{q_1}-1+\alpha_1=\frac{r_1-u_1}{u_1},$$
		we have $\zeta_k\in A_\infty$. Formula (2) gives us  
		\begin{align}
			\zeta_{k-1}&=(\nu_1^{1+\alpha_1})^{\frac{u_2}{r_1-u_2}}(\zeta_k^{1-(1+\alpha_1)\frac{u_2}{r_1}})^{\frac{r_1}{r_1-u_2}}\nonumber\\
			&=(\lambda_1^{-\frac{r_1q_1}{r_1-q_1}\frac{r_1-q_1}{r_1q_1k_1}}\omega_1^{\frac{r_1p_1}{r_1-p_1}\frac{r_1-p_1}{r_1p_1k_1}})^{\frac{u_2}{r_1-u_2}}(\zeta_k^{1-(1+\alpha_1)\frac{u_2}{r_1}})^{\frac{r_1}{r_1-u_2}}\nonumber,
		\end{align}
		since 
		$$-\frac{r_1-q_1}{r_1q_1k_1}+\frac{r_1-p_1}{r_1p_1k_1}=\frac{1}{p_1k_1}-\frac{1}{q_1k_1}=\alpha_1,$$
		and
		$$\frac{\alpha_1u_2}{r_1-u_2}+\frac{r_1}{r_1-u_2}-(1+\alpha_1)\frac{u_2}{r_1-u_2}=1,$$
		we have $\zeta_{k-1}\in A_\infty$. Using the same method, we get $\zeta_2,\dots,\zeta_{k-2}\in A_{\infty}$. 
		Now, we use Lemma \ref{lem 13} $(k-1)$ times to obtain
		\begin{align}
			\|h_{k}&\|_{L^{u_1/r_1}(\lambda_1^{u_1})}=\|\mathcal{A}^{\zeta_{k}}_{\alpha_1,\mathcal{S}^\prime}(h_{k-1}\zeta_{k}^{-(1+\alpha_1)})\nu_1^{1+\alpha_1}\zeta_{k}\|_{L^{u_1/r_1}(\lambda_1^{u_1})}\nonumber\\
			&=\|\mathcal{A}^{\zeta_{k}}_{\alpha_1,\mathcal{S}^\prime}(h_{k-1}\zeta_{k}^{-(1+\alpha_1)})\|_{L^{u_1/r_1}(\zeta_{k})}
			\lesssim \|h_{k-1}\zeta_{k}^{-(1+\alpha_1)}\|_{L^{u_2/r_1}(\zeta_{k})}\nonumber\\
			&\lesssim\cdots
			\lesssim\|h_1\zeta_2^{-(1+\alpha_1)}\|_{L^{u_{{k_1}}}(\zeta_2)}\nonumber\\
			&=\|\big(\mathcal{A}^{\eta_1}_{\mathcal{S}^\prime}(f_1^{r_1}\cdot\eta_1^{-1})^{2-\gamma}\mathcal{A}^{\eta_2}_{\alpha_1,\mathcal{S}^\prime}(h_0\cdot\eta_2^{-(1+\alpha_1)})^{\gamma-1}\nu_1^{1+\alpha_1}\eta_1^{2-\gamma}\eta_2^{\gamma-1}\zeta_2^{-(1+\alpha_1)}\|_{L^{u_k/r_1}(\zeta_2)}\nonumber.
		\end{align}
		For $$\frac{1}{u_k}=\frac{1}{q_1}+(r_1k_1-(\gamma-1))\frac{\alpha_1}{r_1}=(2-\gamma)\frac{1}{p_1}+(\gamma-1)\frac{1}{v},$$ 
		where ${1}/{v}:={1}/{p_1}-{\alpha_1}/{r_1}$, we use H\"older's inequality
		\begin{align}
			\|&\mathcal{A}^{\eta_1}_{\mathcal{S}^\prime}(f_1^{r_1}\cdot\eta_1^{-1})^{2-\gamma}\mathcal{A}^{\eta_2}_{\alpha_1,\mathcal{S}^\prime}(h_0\cdot\eta_2^{-(1+\alpha_1)})^{\gamma-1}\nu_1^{1+\alpha_1}\eta_1^{2-\gamma}\eta_2^{\gamma-1}\zeta_2^{-(1+\alpha_1)}\|_{L^{u_k/r_1}(\zeta_2)}\nonumber\\
			\leqslant&\|\mathcal{A}^{\eta_1}_{\mathcal{S}^\prime}(f_1^{r_1}\cdot\eta_1^{-1})\|_{L^{p_1/r_1}(\eta_1)}^{2-\gamma}\|\mathcal{A}^{\eta_2}_{\alpha_1,\mathcal{S}^\prime}(h_0\cdot\eta_2^{-(1+\alpha_1)})\lambda_1^{\frac{1}{k_1}}\omega_1^{\frac{k}{k_1(\gamma-1)}}\eta_2\eta_1^{(\frac{2-\gamma}{\gamma-1})(1-\frac{r_1}{p_1})}\|_{L^{\nu/r_1}}^{\gamma-1}.\nonumber
		\end{align}
		For the first term, let $\eta_1=\omega_1^{\frac{p_1r_1}{r_1-p_1}}\in A_{\infty}$, and use the $L^{p_1}(\eta_1)$ boundedness of $\mathcal{A}^{\eta_1}_{\mathcal{S}^\prime}$ we have
		$$\|\mathcal{A}^{\eta_1}_{\mathcal{S}^\prime}(f_1^{r_1}\cdot\eta_1^{-1})\|_{L^{p_1/r_1}(\eta_1)}^{2-\gamma}\leqslant\|f_1\omega_1\|_{L^{p_1}}^{r_1(2-\gamma)}.$$
		For the second term, let 
		$$\Big(\lambda_1^{\frac{1}{k_1}}\omega_1^{\frac{k}{k_1(\gamma-1)}}\eta_2\eta_1^{(\frac{2-\gamma}{\gamma-1})(1-\frac{r_1}{p_1})}\Big)^{\frac{\nu}{r_1}}=\eta_2,$$
		we have
		\begin{align}
			\eta_2=\lambda_1^{\frac{\nu}{k_1(r_1-\nu)}}\omega_1^{(r_1-\frac{1}{k_1})\frac{\nu}{r_1-\nu}}=\Big(\lambda_1^{\frac{r_1q_1}{r_1-q_1}(\frac{r_1-q_1}{r_1q_1k_1})}\omega_1^{\frac{r_1p_1}{r_1-p_1}\frac{r_1-p_1}{r_1p_1}(r_1-\frac{1}{k_1})}\Big)^{\frac{\nu}{r_1-\nu}},\nonumber
		\end{align}
		since
		$$\frac{1}{q_1k_1}-\frac{1}{r_1k_1}+\frac{r_1-p_1}{p_1}-\frac{1}{p_1k_1}+\frac{1}{k_1r_1}=\frac{r_1}{p_1}-1-\alpha_1=\frac{r_1-\nu}{\nu},$$
		 then $\eta_2\in A_{\infty}$.
		Let 
		$$(\eta_2^{1+\alpha_1}\nu_1^{1+\alpha_1}\zeta_1)^{\frac{p_1}{r_1}}\eta_2=\zeta_1,$$
		we have $\zeta_1=\omega_1^{\frac{p_1r_1}{r_1-p_1}}\in A_{\infty}$. Thus
		\begin{align}
		\|\mathcal{A}&^{\eta_2}_{\alpha_1,\mathcal{S}^\prime}(h_0\cdot\eta_2^{-(1+\alpha_1)})\lambda_1^{\frac{1}{k_1}}\omega_1^{\frac{k}{k_1(\gamma-1)}}\eta_2\eta_1^{(\frac{2-\gamma}{\gamma-1})(1-\frac{r_1}{p_1})}\|_{L^{\nu/r_1}}^{\gamma-1}\nonumber\\
		=&\|\mathcal{A}^{\eta_2}_{\alpha_1,\mathcal{S}^\prime}(h_0\cdot\eta_2^{-(1+\alpha_1)})\|_{L^{\nu/r_1}(\eta_2)}^{\gamma-1}
		\lesssim \|h_0\cdot\eta_2^{-(1+\alpha_1)}\|_{L^{p_1/r_1}(\eta_2)}^{\gamma-1}\nonumber\\
		=&\|\mathcal{A}^{\zeta_1}_{\mathcal{S}^\prime}(f_1^{r_1}\cdot \zeta_1^{-1})\|_{L^{p_1/r_1}(\zeta_1)}^{\gamma-1}
		\lesssim\|f_1\omega_1\|_{L^{p_1}}^{r_1(\gamma-1)}.\nonumber
		\end{align}
		Combining the arguments in the above
		$$\|h_k^{\frac{1}{r_1}}\|_{L^{u_1}(\lambda_1^{u_1})}^{r_1}=\|h_{k}\|_{L^{u_1/r_1}(\lambda_1^{u_1})}\lesssim \|f_1\omega_1\|^{r_1}_{L^{p_1}}.$$
			\textbf{The case  $q_1<p_1$.} By Lemma \ref{key lem3}, there exists a sparse collection $\mathcal{S}\subset \mathcal{S}^{\prime}\subset  \mathscr{D}$ such that for any $Q\in \mathcal{S}$,
		\begin{align}
			\big\langle {|b-\langle{b}\rangle_{Q}|^{k_1}|f_1|}\big\rangle_{r_1,Q}^{r_1} &\lesssim \frac{1}{|Q|}\int_{Q} \Big(\sum_{P\in \mathcal{S}^\prime:P\subseteq Q}\frac{1}{|P|}\int_{P}|b-\langle b \rangle_{P}|\chi_{P}\Big)^{r_1k_1}|f_1|^{r_1}\nonumber\\
			&= \frac{1}{|Q|}\int_{Q} \Big(\sum_{P\in \mathcal{S}^\prime:P\subseteq Q}\frac{\nu_1(P)}{|P|}\Omega_{\nu_1}(b,P)\chi_P\Big)^{r_1k_1}|f_1|^{r_1}\nonumber.
		\end{align}
	  Let $k:=\lfloor k_1r_1 \rfloor $ and $\gamma:=r_1k_1-(k-1)$. Applying Lemma \ref{lem 11} $(k-1)$ times, we have
		\begin{align}
			\int_{Q}& \Big(\sum_{P\in \mathcal{S}^\prime:P\subseteq Q}\frac{\nu_1(P)}{|P|}\Omega_{\nu_1}(b,P)\chi_P\Big)^{r_1k_1}|f_1|^{r_1}\nonumber\\
			&\lesssim\sum_{\substack{P_{k-1}\subseteq \cdots \subseteq P_1\subseteq Q\\P_1,\dots,P_{k}\in\mathcal{S}^\prime}}\frac{\nu_1(P_1)}{|P_1|}\Omega_{\nu_1}(b,P_1)\cdots\frac{\nu_1(P_{k-1})}{|P_{k-1}|}\Omega_{\nu_1}(b,P_{k-1})\nonumber\\
			&\quad\quad\quad\cdot\int_{P_{k-1}}\Big(\sum_{P_{k}\subseteq P_{k-1}} \frac{\nu_1(P_{k})}{|P_{k}|}\Omega_{\nu_1}(b,P_{k})\chi_{P_{k}}\Big)^{\gamma}|f_1|^{r_1}\nonumber.
		\end{align}
			Then, we denote
		$$\mathcal{A}^{\zeta}_{\mathcal{S}^\prime}(\varphi):=\sum_{Q \in \mathcal{S}^\prime}\langle |\varphi| \rangle^{\zeta}_{Q}\chi_{Q},\quad \mathcal{A}^{\zeta}_{\mathcal{S}^\prime,b}(\varphi):=\mathcal{A}^{\zeta}_{\mathcal{S}^\prime}(\varphi)M_{\nu_1}^\#(b).$$
	   Similar to the previous method. Let $\zeta_1,\eta_1,\eta_2\in A_{\infty}$, which will be determined later. Using Lemma \ref{lem 12}, Lemma \ref{key lem1} and Minkowski's inequality, we obtain
		\begin{align}
			\int_{P_{k-1}}&\Big(\sum_{P_{k}\subseteq P_{k-1}} \frac{\nu_1(P_{k})}{|P_{k}|}\Omega_{\nu_1}(b,P_k)\chi_{P_{k}}\Big)^{\gamma}|f_1|^{r_1}\nonumber\\
			\simeq&\sum_{P_{k}\subseteq P_{k-1}}\frac{\nu_1(P_{k})}{|P_{k}|}\Omega_{\nu_1}(b,P_k)\Big(\int_{P_k}|f_1|^{r_1}\Big)^{2-\gamma}\Big(\sum_{P\subseteq P_{k}}\frac{\nu_1(P)}{|P|}\Omega_{\nu_1}(b,P)\int_{P}|f_1|^{r_1}\Big)^{\gamma-1}\nonumber\\
			\lesssim& \sum_{P_{k}\subseteq P_{k-1}}\frac{\nu_1(P_{k})}{|P_{k}|}\Omega_{\nu_1}(b,P_k)\Big(\int_{P_k}|f_1|^{r_1}\Big)^{2-\gamma}\Big(\int_{P_k}\mathcal{A}_{\mathcal{S}^{\prime},b}^{\zeta_1}(f^{r_1}\zeta_1^{-1})\zeta_1\nu_1\Big)^{\gamma-1}\nonumber\\
			:=& \sum_{P_{k}\subseteq P_{k-1}}\frac{\nu_1(P_{k})}{|P_{k}|}\Omega_{\nu_1}(b,P_k)\Big(\int_{P_k}|f_1|^{r_1}\Big)^{2-\gamma}\Big(\int_{P_k}h_0\Big)^{\gamma-1}\nonumber.
			\end{align}
	    	By a similar method as above and H\"older's inequality	
	     \begin{align}
	     	& \sum_{P_{k}\subseteq P_{k-1}}\frac{\nu_1(P_{k})}{|P_{k}|}\Omega_{\nu_1}(b,P_k)\Big(\int_{P_k}|f_1|^{r_1}\Big)^{2-\gamma}\Big(\int_{P_k}h_0\Big)^{\gamma-1}\nonumber\\
	     	&\quad=\sum_{P_{k}\subseteq P_{k-1}}\frac{\nu_1(P_{k})\eta_1(P_k)^{2-\gamma}\eta_2(P_k)^{\gamma-1}}{|P_{k}|}\Omega_{\nu_1}(b,P_k)\Big(\Big\langle|f_1|^{r_1}\eta_1^{-1}\Big\rangle^{\eta_1}_{P_k}\Big)^{2-\gamma}\Big(\Big\langle|h_0|\eta_2^{-1}\Big\rangle^{\eta_1}_{P_k}\Big)^{\gamma-1}\nonumber\\
	     	&\quad\lesssim\sum_{P_{k}\subseteq P_{k-1}}\Omega_{\nu_1}(b,P_k)\Big(\Big\langle|f_1|^{r_1}\eta_1^{-1}\Big\rangle^{\eta_1}_{P_k}\Big)^{2-\gamma}\Big(\Big\langle|h_0|\eta_2^{-1}\Big\rangle^{\eta_1}_{P_k}\Big)^{\gamma-1}\nu_1\eta_1^{2-\gamma}\eta_2^{\gamma-1}(P_k)\nonumber\\
	     	&\quad\leq\int_{P_{k-1}}\sum_{P_{k}\subseteq P_{k-1}}\Omega_{\nu_1}(b,P_k)\Big(\Big\langle|f_1|^{r_1}\eta_1^{-1}\Big\rangle^{\eta_1}_{P_k}\Big)^{2-\gamma}\Big(\Big\langle|h_0|\eta_2^{-1}\Big\rangle^{\eta_1}_{P_k}\Big)^{\gamma-1}\chi_{P_k}\nu_1\eta_1^{2-\gamma}\eta_2^{\gamma-1}\nonumber\\
	     	&\quad\leq\int_{P_{k-1}}\mathcal{A}_{\mathcal{S}^\prime}^{\eta_1}(f_1^{r_1}\cdot\eta_1^{-1})^{2-\gamma}\mathcal{A}_{\mathcal{S}^\prime}^{\eta_2}(h_0\cdot\eta_2^{-1})^{\gamma-1}M_{\nu_1}^\#(b)\nu_1\eta_1^{2-\gamma}\eta_2^{\gamma-1}
	     	:=\int_{P_{k-1}}h_1\nonumber.	
	     \end{align}
	     for $i=2,\dots,k$, define
		$$h_i:=\mathcal{A}^{\zeta_i}_{\mathcal{S}^\prime}(h_{i-1}\zeta_i^{-1})M_{\nu_1} ^\#(b)\nu_1\zeta_i,$$
		and let $\zeta_2,\dots,\zeta_{k}\in A_{\infty}$ which will be determined later.
		\begin{align}
			\sum_{P_{k-1}\subseteq P_{k-2}}&\frac{\nu_1(P_{k-1})}{|P_{k-1}|}\Omega_{\nu_1}(b,P_{k-1})\int_{P_{k-1}}h_1\nonumber\\
			\lesssim&\int_{P_{k-2}} \sum_{P_{k-1}\subseteq P_{k-2}}\langle h_1\zeta_2^{-1}\rangle_{P_{k-1}}^{\zeta_2}\chi_{P_{k-1}}M_{\nu_1}^\#(b)\nu_1\zeta_2:=\int_{P_{k-1}}h_2\nonumber.
		\end{align}
		 By iterating this process we arrive at
		$$\int_{Q} \Big(\sum_{P\in \mathcal{S}^\prime:P\subseteq Q}\frac{\nu_1(P)}{|P|}\Omega_{\nu_1}(b,P)\chi_P\Big)^{\gamma}|f_1|^{r_1}=\frac{1}{|Q|}\int_{Q}h_{k}.$$
		Now, we conclude
		$$\sum_{Q\in \mathcal{S}}\big\langle {|b-\langle{b}\rangle_{Q}|^{k_1}|f_1|}\big\rangle_{r_1,Q}\langle |f_2|\rangle_{r_2,Q}\langle |g|\rangle_{s^\prime,Q}|Q|
		\lesssim\sum_{Q\in \mathcal{S}} \langle h_{k}^{\frac{1}{r_1}}\rangle_{r_1,Q}\langle f_2 \rangle_{r_2,Q} \langle g\rangle_{s^\prime,Q} |Q|,$$
		by Theorem \ref{mmmain thm}, we have
		\begin{align}
			\sum_{Q\in \mathcal{S}}&\big\langle {|b-\langle{b}\rangle_{Q}|^{k_1}|f_1|}\big\rangle_{r_1,Q}\langle |f_2|\rangle_{r_2,Q}\langle |g|\rangle_{s^\prime,Q}|Q|\nonumber\\
			&\quad\quad\lesssim\|h_{k}^{\frac{1}{r_1}}\lambda_1\|_{L^{q_1}}\|f_2\omega_2\|_{L^{p_2}}\|g(\lambda_1\omega_2)^{-1}\|_{L^{q^\prime}}\nonumber,
		\end{align}
	For term $\|h_{k}\|_{L^{q_1/r_1}(\lambda_1^{q_1})}$,  define 
	$$\frac{1}{u_i}:=\frac{1}{q_1}-\frac{i}{t_1r_1}=\frac{i}{r_1k_1p_1}+(1-\frac{i}{r_1k_1})\frac{1}{q_1},\quad i=1,\dots,k,$$
	and we need
		\begin{enumerate}
		\item  $(\nu_1^{1+\alpha_1}\zeta_k)^{\frac{u_1}{r_1}}\lambda_1^{u_1}=\zeta_k$;
		\item 
		$(\nu_1^{1+\alpha_1}\zeta_{i-1}\zeta_i^{-(1+\alpha_1)})^{\frac{u_{k+2-i}}{r_1}}\zeta_i=\zeta_{i-1},\quad i=3,\dots,k$.
	\end{enumerate}
	The same as the previous case. These give us $\zeta_i\in A_{\infty}$, 
	for $i=2,\dots,k$. Applying H\"older's inequality and $L^{u_i}(\zeta_k)$ boundedness of $\mathcal{A}^{\zeta_k}_{\mathcal{S}^\prime}$, we get
		\begin{align}
			\|h_{k}\|_{L^{q_1/r_1}(\lambda_1^{q_1})}&=\|\mathcal{A}^{\zeta_{k}}_{\mathcal{S}^\prime}(h_{k-1}\zeta_{k}^{-1})M_{\nu_1} ^\#(b)\nu_1\zeta_{k}\|_{L^{q_1/r_1}(\lambda_1^{q_1})}\nonumber\\
			&\leq \|M_{\nu_1} ^\#(b)\|_{L^{t_1}(\nu_1)}\|\mathcal{A}^{\zeta_{k}}_{\mathcal{S}^\prime}(h_{k-1}\zeta_{k}^{-1})\|_{L^{u_1/r_1}(\zeta_{k})}\nonumber\\
			&\leq \cdots
			\leq\|M_{\nu_1} ^\#(b)\|_{L^{t_1}(\nu_1)}^{k-1}\|h_1\zeta_2^{-1}\|_{L^{u_{k-1}}(\zeta_2)}\nonumber.
		\end{align}
		Now, define ${1}/{v}={1}/{p_1}+{1}/{r_1t_1}$ and
		$$\frac{1}{u_{k-1}}=\frac{1}{q_1}-\frac{k-1}{r_1t_1}=\frac{1}{r_1t_1}+(2-\gamma)\frac{1}{p_1}+(\gamma-1)\frac{
		1}{v},$$
		we can estimate by H\"older's inequality,
		\begin{align}
		\|h_1&\zeta_2^{-1}\|_{L^{u_{k-1}}(\zeta_2)}\nonumber\\
		=&\|\mathcal{A}_{\mathcal{S}^\prime}^{\eta_1}(f_1^{r_1}\cdot\eta_1^{-1})^{2-\gamma}\mathcal{A}_{\mathcal{S}^\prime}^{\eta_2}(h_0\cdot\eta_2^{-1})^{\gamma-1}M_{\nu_1}^\#(b)\nu_1\eta_1^{2-\gamma}\eta_2^{\gamma-1}\zeta_2^{-1}\zeta_2^{\frac{r_1}{u_{k-1}}}\|_{L^{u_{k-1}/r_1}}\nonumber\\
		\leqslant& \|M_{\nu_1}^\#(b)\|_{L^{t_1}(\nu_1)}\|\mathcal{A}_{\mathcal{S}^\prime}^{\eta_1}(f_1^{r_1}\cdot\eta_1^{-1})\|^{2-\gamma}_{L^{p_1/r_1}(\eta_1)}\nonumber\\
		&\quad\quad\quad\quad\quad\cdot\|\mathcal{A}_{\mathcal{S}^\prime}^{\eta_2}(h_0\cdot\eta_2^{-1})\nu_1^{(1+\alpha_1)\frac{1}{\gamma-1}}\eta_1^{(\frac{2-\gamma}{\gamma-1})(1-\frac{r_1}{p_1})}\eta_2\zeta_2^{(\frac{r_1}{u_{k-1}}-1)\frac{1}{\gamma-1}}\|_{L^{v/r_1}}^{\gamma-1}\nonumber.
		\end{align}
		For the second term, let $\eta_1=\omega_1^{\frac{p_1r_1}{r_1-p_1}}\in A_{\infty}$, we have
		$$\|\mathcal{A}^{\eta_1}_{\mathcal{S}^\prime}(f_1^{r_1}\cdot\eta_1^{-1})\|_{L^{p_1/r_1}(\eta_1)}^{2-\gamma}\leqslant\|f_1\omega_1\|_{L^{p_1}}^{r_1(2-\gamma)}.$$
		For the third term, let 
		$$\Big(\nu_1^{(1+\alpha_1)\frac{1}{\gamma-1}}\eta_1^{(\frac{2-\gamma}{\gamma-1})(1-\frac{r_1}{p_1})}\eta_2\zeta_2^{(\frac{r_1}{u_{k-1}}-1)\frac{1}{\gamma-1}}\Big)^{\frac{\nu}{r_1}}=\eta_2,$$
		that is
		$$\eta_2=\lambda_1^{\frac{\nu}{k_1(r_1-\nu)}}\omega_1^{(r_1-\frac{1}{k_1})\frac{\nu}{r_1-\nu}}\in A_{\infty}.$$
		Let $\zeta_1=\omega_1^{\frac{p_1r_1}{r_1-p_1}}\in A_{\infty}$, we have
		\begin{align}
			\|\mathcal{A}&_{\mathcal{S}^\prime}^{\eta_2}(h_0\cdot\eta_2^{-1})\nu_1^{(1+\alpha_1)\frac{1}{\gamma-1}}\eta_1^{(\frac{2-\gamma}{\gamma-1})(1-\frac{r_1}{p_1})}\eta_2\zeta_2^{(\frac{r_1}{u_{k-1}}-1)\frac{1}{\gamma-1}}\|_{L^{v/r_1}}^{\gamma-1}\nonumber\\
			=&\|\mathcal{A}_{\mathcal{S}^\prime}^{\eta_2}(h_0\cdot\eta_2^{-1})\|_{L^{v/r_1}(\eta_2)}^{\gamma-1}
			\lesssim \|\mathcal{A}_{\mathcal{S}^{\prime},b}^{\zeta_1}(f^{r_1}\zeta_1^{-1})\zeta_1\nu_1\eta_2^{-1}\|_{L^{v/r_1}(\eta_2)}^{\gamma-1}\nonumber\\
			\leqslant& \|M_{\nu_1}^\#(b)\|_{L^{t_1}(\nu_1)}^{\gamma-1}\|\mathcal{A}_{\mathcal{S}^\prime}^{\zeta_1}(f_1^{r_1}\cdot\zeta_1^{-1})\omega_1^{r_1}\zeta_1\|_{L^{p_1/r_1}}^{\gamma-1}	\leqslant\|M_{\nu_1}^\#(b)\|_{L^{t_1}(\nu_1)}^{\gamma-1}\|f_1\omega_1\|_{L^{p_1}}^{r_1(\gamma-1)}.\nonumber
		\end{align}
		Combining the arguments in the above
		\begin{align}
			\|h_{k}^{\frac{1}{r_1}}\|_{L^{q_1}(\lambda_1^{q_1})}^{r_1}=&\|h_k\|_{L^{q_1/r_1}(\lambda_1^{q_1})}\nonumber\\
			\lesssim&\|M_{\nu_1}^\#(b)\|_{L^{t_1}(\nu_1)}^{r_1k_1}\|f_1\omega_1\|^{r_1}_{L^{p_1}}\nonumber,
		\end{align}
		and we have done.
	\end{proof}
	\begin{proof}[Proof of Theorem \ref{thm 10}] 
		Combining Theorem \ref{main thm} and Theorem \ref{mmain thm}, we arrive at Theorem \ref{thm 10}. The second summation term in (\ref{key formular6})
		$$\sum_{Q\in \mathcal{S}} \langle |f_1|\big\rangle_{r_1,Q}\big\langle |f_2|\big\rangle_{r_2,Q}\big\langle {|b-\langle{b}\rangle_Q|^{k_1}|g|}\big\rangle_{s^{\prime},Q}\big|Q|,$$
		can be handled similarly to Theorem \ref{mmain thm}. 
    \end{proof}
	
	\textbf{Acknowledgment.} 
	The author is supported by the National Key R\&D Program of China (Grant No. 2021YFA1002500). The author thanks Kangwei Li for his valuable  suggestions and Linfei Zheng  for his helpful discussion
	\bibliographystyle{plain}
	\bibliography{ref}
\end{document}